\documentclass[a4paper, 11pt, headings = small, abstract,numbers=endperiod]{scrartcl}
\usepackage[english]{babel}
\usepackage{amscd,amsmath,amsthm,array,hhline,mathrsfs, enumerate, amssymb, booktabs,fancybox,calc,textcomp,xcolor,graphicx,bbm,xspace,nicefrac,stmaryrd,url,arcs,listings,yhmath}
\usepackage[title]{appendix}
\usepackage[semibold]{libertine}
\usepackage[upint]{libertinust1math}
\usepackage[scr=boondoxupr, scrscaled = 1.0]{mathalpha}

\usepackage[scaled=0.95]{inconsolata}
\usepackage{accents}

\DeclareMathAlphabet{\mathsf}{OT1}{\sfdefault}{m}{n}
\SetMathAlphabet{\mathsf}{bold}{OT1}{\sfdefault}{b}{n}
\DeclareMathAlphabet{\mathfrak}{U}{jkpmia}{m}{it}
\SetMathAlphabet{\mathfrak}{bold}{U}{jkpmia}{bx}{it}
\usepackage[utf8]{inputenc}

\usepackage{enumitem}
\usepackage{etoolbox}
\usepackage[normalem]{ulem}
\usepackage{mathtools}
\usepackage{geometry}
\usepackage{float}

\usepackage{bm}
\usepackage{tikz}
\usepackage[outdir =./]{epstopdf}

\numberwithin{equation}{section}

\usepackage{chngcntr}
\counterwithin{figure}{section}

\usepackage{xcolor}
\usepackage[colorlinks=true, allcolors = myteal]{hyperref}

\definecolor{myteal}{RGB}{0 123 137}
\newcommand{\e}{\mathrm{e}}
\renewcommand{\hat}{\widehat}
\renewcommand{\tilde}{\widetilde}
\usepackage[numbers]{natbib}
\usepackage{dsfont}
\usepackage{enumitem}
\usepackage[title]{appendix}

\setkomafont{sectioning}{\normalcolor\bfseries}
\setkomafont{descriptionlabel}{\normalcolor\bfseries}
\setkomafont{section}{\large}
\setkomafont{subsection}{\normalsize}
\setkomafont{subsubsection}{\normalsize}
\setkomafont{paragraph}{\normalsize}
\setkomafont{subparagraph}{\normalsize}
\setkomafont{author}{\large}
\setkomafont{title}{\large\bfseries}
\setkomafont{date}{\normalsize}

\usepackage{graphicx}
\usepackage{subcaption}

\theoremstyle{plain}
\newtheorem{theorem}{Theorem}
\newtheorem{lemma}[theorem]{Lemma}

\newtheorem{corollary}[theorem]{Corollary}
\newtheorem{proposition}[theorem]{Proposition}

\theoremstyle{definition}
\newtheorem{remark}[theorem]{Remark}
\newtheorem{example}[theorem]{Example}

\newtheorem{innercustomgeneric}{\customgenericname}
\providecommand{\customgenericname}{}
\newcommand{\newcustomtheorem}[2]{%
	\newenvironment{#1}[1]
	{%
		\renewcommand\customgenericname{#2}%
		\renewcommand\theinnercustomgeneric{##1}%
		\innercustomgeneric
	}
	{\endinnercustomgeneric}
}

\newcustomtheorem{customass}{Assumption}

\numberwithin{equation}{section}
\numberwithin{theorem}{section}

\newcommand{\sc}[2]{\langle#1,#2\rangle}
\newcommand{\norm}[1]{\lVert#1\rVert}
\let\P\relax
\DeclareMathOperator{\P}{{\mathbb P}}
\DeclareMathOperator{\E}{{\mathbb E}}
\DeclareMathOperator{\R}{{\mathbb R}}

\renewcommand{\theta}{\vartheta}    
\newcommand*\diff{\mathop{}\!\mathrm{d}}

\allowdisplaybreaks

\makeatother
\title{\fontsize{16}{19} \selectfont Nonparametric velocity estimation in stochastic convection-diffusion equations from multiple local measurements}
\author{Claudia Strauch\thanks{Aarhus University, Department of Mathematics, Ny Munkegade 118, 8000 Aarhus C, Denmark. \newline Email: \href{mailto:strauch@math.au.dk}{strauch@math.au.dk}/\href{mailto:tiepner@math.au.dk}{tiepner@math.au.dk}} \and Anton Tiepner\footnotemark[1]}
\begin{document}

\maketitle
\begin{abstract}
We investigate pointwise estimation of the function-valued velocity field of a second-order linear SPDE. Based on multiple spatially localised measurements, we construct a weighted augmented MLE and study its convergence properties as the spatial resolution of the observations tends to zero and the number of measurements increases. By imposing Hölder smoothness conditions, we recover the pointwise convergence rate known to be minimax-optimal in the linear regression framework. The optimality of the rate in the current setting is verified by adapting the lower bound ansatz based on the RKHS of local measurements to the nonparametric situation.
 \end{abstract}

\noindent
\small{\textit{\href{https://mathscinet.ams.org/mathscinet/msc/msc2020.html}{2020 MSC:}} Primary 60H15; secondary 62G05, 62G08\\
\textit{key words:} Stochastic convection-diffusion equation, nonparametric regression, local likelihood estimation, minimax lower bound, local measurements.}
\normalsize

\section{Introduction}
Stochastic partial differential equations (SPDEs) are an appealing tool to model spatio-temporal data. 
They describe the evolution of dynamical systems and can be utilised in almost all areas of natural sciences, finance, economics, and many more applied disciplines. 
By including random forcing terms, SPDEs also account for microscopic scaling limits or model misspecification.
We will focus on the important subclass of \emph{stochastic convection-diffusion} or \emph{advection-diffusion equations} which can also serve as a basis for more complex models. 
They describe the movement of quantities (such as particles, heat, energy, etc.) in a physical system and find applications in, but are not limited to, weather forecasts \cite{sigrist_stochastic_2015,sigrist_dynamic_2012, liu_statistical_2021}, neuronal responses \cite{walsh_stochastic_1981,Tuckwell2013}, solar radiation \cite{clarotto_2022}, air quality \cite{liu_statistical_2016}, sediment concentrations \cite{stroud_2010}, biomass distributions \cite{denaro_2013}, groundwater flows \cite{serrano_1990}, and term structure movements \cite{cont_modeling_2005}.

More specifically, for a finite time horizon $T$, we consider the solution $X=(X(t))_{0\leq t\leq T}$ to the linear parabolic SPDE
\begin{equation} \label{eq: SPDE}
	\begin{cases} \diff X(t)=A_\vartheta X(t) \diff t+ \diff W(t), & t\in (0,T],\\ X(0)=X_0\in L^2(\Lambda), &\\ 
		X(t)|_{\partial\Lambda}=0,& t \in (0,T],\end{cases}
\end{equation}
on a bounded open domain $\Lambda\subset\R^d$ with $C^2$-boundary $\partial\Lambda$, Dirichlet boundary conditions, driven by a cylindrical Brownian motion $W=(W(t))_{0\le t\le T}$. 
The second-order elliptic operator $A_\theta$ appearing in \eqref{eq: SPDE} is specified as
\begin{equation}\label{eq: opA}
	A_\theta z=a\Delta z+\theta\cdot\nabla z+cz,
\end{equation}
where $a>0$, $\theta$ and $c$ represent the (constant) diffusivity, the velocity field and the reaction coefficient, inducing spatial diffusion, transportation and damping, respectively. 
While the analytical theory of SPDEs is well understood and established, see, e.g., \cite{da_prato_stochastic_2014,lototsky_SPDEs_2017,liu_stochastic_2015, hairer_introduction_2009}, the literature on their statistical aspects is somewhat limited, and many research questions are still open.
As a concrete example, to the best of our knowledge, estimation of a function-valued velocity has not yet been investigated. 
We want to fill in this gap by estimating the function-valued velocity field $\theta$ by means of nonparametric methods based on local measurements. 

Parameter estimation for SPDEs is widely studied in the literature, but primarily devoted to a scalar parameter in $A_\theta=\theta A+B$ for some (non-) linear operators $A$ and $B$. 
When $A$ and $B$ share a common system of eigenvectors and are self-adjoint, \cite{huebner_asymptotic_1995} constructed a maximum likelihood estimator (MLE) for $\theta$, relying on $N$ spectral measurements $(\sc{X(t)}{e_i})_{0\leq t\leq T}$, $i=1,\dots,N$, with an eigenbasis $(e_i)_{i\geq1}$.
Given some relation between the differential order of $A$ and $B$ and the dimension $d$, the derived MLE was shown to be consistent and asymptotically normal. This so-called spectral approach was subsequently adapted and extended to different settings, such as nonlinear SPDEs \cite{cialenco_parameter_2011,pasemann_drift_2020}, fractional noise \cite{cialenco_asymptotic_2009}, or joint parameter estimation \cite{lototsky_parameter_2003}. However, the majority of these studies considered the case where $\theta$ specifies the highest order operator, i.e., $\operatorname{ord}(A)>\operatorname{ord}(B)$, and there is no known estimator for a constant transport coefficient $\theta$ in \eqref{eq: opA} in the spectral approach setting.
Based on discrete observations $X(t_j,x_i)$ in time and space, \cite{hildebrandt_parameter_2019,kaino_parametric_2021,tonaki2022parameter} analysed power variations and contrast estimators in dimension one and two for all occurring quantities in the parametric version of \eqref{eq: opA}. Reaction or source-sink terms have been studied, for example, by \cite{hildebrandt2021nonparametric,gaudlitz2022estimation,ibragimov_2001}. We refer to \cite{cialenco_statistical_2018} for a detailed overview of further related literature.

In this paper, we construct a pointwise estimator $\hat\theta(x)$ for the velocity field $\theta$, evaluated at the spatial location $x\in\Lambda$ from local measurement processes for multiple locations, i.e., our data are given by observing the solution to \eqref{eq: SPDE} locally in space and continuously in time. 
Given some fixed function $K\in H^2(\R^d)$ with compact support, we consider points $x_1,\dots,x_N\in\Lambda$ and a resolution level $\delta>0$ small enough such that the localised functions $K_{\delta,x_k}(x)=\delta^{-d/2}K(\delta^{-1}(x-x_k))$, $k=1,\ldots,N$, are supported in $\Lambda$. 
In optical systems, they are known as \emph{point spread functions} \cite{aspelmeier_modern_2015,backer_extending_2014}, and they describe the physical limitation that $X(t_j,x_i)$ can only be measured up to some locally blurred average, i.e., a convolution with the point spread function.
Specifically, the local measurements of $X$ are given as the continuously observed processes $X_{\delta},X_{\delta}^\Delta\in L^2([0,T];\R^N)$ and $X_{\delta}^\nabla\in L^2([0,T];\R^{d\times N})$, where, for $k=1,\ldots,N$,
\begin{equation}\begin{split}\label{eq: locob}
		(X_\delta)_k&=X_{\delta,k}=(\sc{X(t)}{K_{\delta,x_k}})_{0\leq t\leq T},\\
		(X_\delta^\nabla)_k&=X^\nabla_{\delta,k}=(\sc{X(t)}{\nabla K_{\delta,x_k}})_{0\leq t\leq T},\\
		(X_\delta^\Delta)_k&=X^\Delta_{\delta,k}=(\sc{X(t)}{\Delta K_{\delta,x_k}})_{0\leq t\leq T}.
	\end{split}
\end{equation}
Local measurements were introduced by \cite{altmeyer_nonparametric_2020}. There, the authors investigated the estimation of a nonparametric diffusivity $a(x)$ and demonstrated that it can already be estimated with the parametric minimax rate $\delta$ upon observing the local information $X_{\delta,x}$. 
The method proved to be robust to low-order nonlinearities, cf.~\cite{altmeyer_parameterSemi_2020,altmeyer_parameter_2020}, and was used in a direct application to cell repolarisation, estimating the diffusivity of the activator in the stochastic Meinhardt model \cite{altmeyer_parameter_2020}. 
Adapting the extended MLE approach of \cite{altmeyer_nonparametric_2020}, \cite{altmeyer_anisotrop2021} have considered the fully anisotropic parametric version of \eqref{eq: opA}, addressing joint estimation of diffusivity, velocity, and reaction components.
In particular, it has been shown that transport and damping coefficients cannot be estimated consistently in finite time if the number $N=N(\delta)$ of local measurements remains finite. 
If the number of measurements is chosen to be maximal, i.e., $N\asymp\delta^{-d}$, the derived convergence rates agree with the convergence rates obtained with the spectral approach of \cite{huebner_asymptotic_1995}. 
In the case of the transport coefficient $\theta$, the convergence rate $N^{-1/2}$ has been proven to be optimal.

Let us briefly describe our main findings. 
We combine the approach of \cite{altmeyer_anisotrop2021} with techniques from nonparametric regression and local likelihood estimation. The contribution of each measurement is individually weighted and controlled by a bandwidth $h=h(\delta)$ to account for bias reduction. The obtained \emph{weighted augmented MLE} $\hat{\theta}_\delta(x)$ is consistent, and under appropriate Hölder smoothness assumptions, it satisfies, for $x\in\Lambda$,
\begin{equation}
	\label{eq: rateintro}
	\hat{\theta}_\delta(x)-\theta(x)=O_{\P}(h^\beta+(Nh^d)^{-1/2}),\quad \beta\in(1,2].
\end{equation}
Optimising \eqref{eq: rateintro} with respect to $h$, we obtain the convergence rate $N^{-\beta/(2\beta+d)}$ known from local linear regression estimation. This convergence rate turns out to be optimal, as we demonstrate by adapting the lower bound ansatz of \cite{altmeyer_anisotrop2021}, which is based on the RKHS of our local measurements and its relation to the Hellinger distance, to the nonparametric framework. 

The paper is structured as follows. We specify the model and construct the estimator in Section \ref{sec: main}.
Section \ref{subs: pointwise} provides upper bounds on the pointwise risk of the estimator, along with a discussion of the involved assumptions and a number of examples and applications.
Lower bounds are stated in Section \ref{sec: optimality}.
All proofs are deferred to Section \ref{sec: proofs}. 

\paragraph{Notation}
Throughout this paper, the time horizon $T<\infty$ is fixed, and we work on a filtered probability space $(\Omega, \mathcal{F},(\mathcal{F}_t)_{0\leq t\leq T},\P)$. 
We write $a\lesssim b$ if $a\leq Mb$ holds for a universal constant $M$ not depending on $\delta$, $N$, $h$, or a spatial location $x\in\Lambda$, and $a\lesssim_sb$ if $a\le Cb$ with a constant $C$ explicitly depending on the quantity $s$. 
Unless otherwise stated, all limits are to be understood as the spatial resolution level $\delta\to 0$. 
For an open set $U\subset\R^d$, $L^2(U)$ is the usual $L^2$ space with inner product $\sc{\cdot}{\cdot}=\sc{\cdot}{\cdot}_{L^2(U)}$. 
The Euclidean inner product and distance of two vectors $a,b\in\R^p$ are denoted by $a\cdot b$ and $|a-b|$, respectively. 
Let $H^k(U)$ denote the usual Sobolev spaces, and denote by $H_0^1(U)$ the completion of $C_c^{\infty}(U)$, the space of smooth compactly supported functions, relative to the $H^1(U)$ norm. 
For a multi-index $\alpha=(\alpha_1,\dots,\alpha_d)$, let $D^{\alpha}$ be the $\alpha$-th weak partial derivative operator of order $|\alpha|=\alpha_1+\dots+\alpha_d$.
The gradient, divergence and Laplace operator are denoted by $\nabla$, $\nabla\cdot$ and $\Delta$, respectively. 
For $\beta>0$, denote by $\mathcal{H}(\beta)$ the space of functions $f\colon \Lambda\to\R$ with continuous derivatives up to order $\lfloor\beta\rfloor$ such that their $\lfloor\beta\rfloor$-th partial derivatives are Hölder continuous with exponent $\beta-\lfloor\beta\rfloor\leq 1$.

\section{Pointwise estimation approach}\label{sec: main}
Our interest is in estimating the velocity coefficient appearing in the second-order linear elliptic differential operator $A_\theta$ as introduced in \eqref{eq: opA} with domain $H^1_0(\Lambda)\cap H^2(\Lambda)$.
For $z\in H^1_0(\Lambda)\cap H^2(\Lambda)$, its adjoint $A_\theta^\ast$ is defined by 
\[
A^\ast_\theta z \coloneqq a\Delta z-\nabla\cdot \theta z+cz=a\Delta z-\theta\cdot\nabla z +(c-\nabla\cdot\theta)z.
\]
Both $A_\theta$ and $A_\theta^\ast$ generate analytic semigroups, denoted by $(S_\theta(t))_{t\geq0}$ and $(S_\theta^\ast(t))_{t\geq0}$, respectively, on $L^2(\Lambda)$.
The weak solution to \eqref{eq: SPDE} is given by
\[
X(t)=S_\theta(t)X_0+\int_0^tS_\theta(t-t')\diff W(t').
\]
As discussed in \cite[Proposition 2.1]{altmeyer_nonparametric_2020}, it only takes values in negative-order Sobolev spaces, but still allows the definition of real-valued centred Gaussian processes $(\sc{X(t)}{z})_{0\leq t\leq T,z\in L^2(\Lambda)}$, satisfying for $z\in H^1_0(\Lambda)\cap H^2(\Lambda)$
\begin{equation}
	\label{eq: weaksol}
	\sc{X(t)}{z}=\sc{X_0}{z}+\int_0^t\sc{X(t')}{A^*_\theta z}\diff t'+\sc{W(t)}{z},\quad 0\leq t\leq T.
\end{equation} 

Our nonparametric analysis relies on Hölder smoothness conditions.
Let $\beta\in(1,2]$. 
We assume that the (possibly unknown) diffusion coefficient $a>0$ is constant, each component $\theta_i\colon\Lambda\to\R$, $i=1,\ldots,d$, of the velocity field $\theta$ is contained in $\mathcal{H}(\beta)$, and the (nuisance) reaction function $c\colon\Lambda\to\R$ belongs to $\mathcal{H}(\beta-1)$. 
Since the differential operator $A_\theta^\ast$ contains the first-order derivative of $\theta$, we require for existence reasons, cf.~also \cite[Proposition 3.5]{altmeyer_nonparametric_2020}, that $\theta$ is continuously differentiable, i.e., $\beta>1$.

Recall that we work in a local measurements framework, i.e., we construct an estimator based on the observations \eqref{eq: locob}. 
Let $W_k(t)\coloneqq \sc{W(t)}{K_{\delta,x_k}}\norm{K}_{L^2(\R^d)}^{-1}$ be scalar Brownian motions. 
Each local measurement forms an Itô process with initial condition $X_{\delta,k}(0)=\sc{X_0}{K_{\delta,x_k}}$ and, using \eqref{eq: weaksol}, 
\begin{equation}
	\label{eq: process xdelta}
	\diff X_{\delta,k}(t)=\sc{X(t)}{A_\theta^\ast K_{\delta,x_k}}\diff t+\norm{K}_{L^2(\R^d)}\diff W_k(t),\quad k=1,\ldots,N.
\end{equation}

Before constructing an estimator for $\theta(x)$, we give a brief recap on the construction of local polynomial log-likelihood estimators. 
The following is based on \cite{loader_1999}. 
We also refer to \cite{Fan_1996,Fan_1998,Ruppert_1994} for further discussion of the local likelihood approach and to the monograph \cite{Wasserman_2006} for an overview of general nonparametric techniques.
Suppose we observe response variables 
\[
Y_i\sim f(\cdot,\mu(x_i)),\quad i=1,\ldots, n,
\]
with density $f$ depending on the design points $x_i\in\R$ via an unknown function $\mu$. 
Simple examples are given by nonparametric regression, where $Y_i=\mu(x_i)+\varepsilon_i$, with $\varepsilon_i\sim\mathcal{N}(0,\sigma^2)$,
or logistic regression, where $P(Y_i=1)=p(x_i)$, $P(Y_i=0)=1-p(x_i)$, and we consider the link function $\mu(x_i)=\log(p(x_i)/(1-p(x_i))$.
Assuming that $\mu(x)$ has a polynomial fit of degree $p\in\mathbb{N}_0$, i.e., by a $p$-th order Taylor approximation,
\[
\mu(x_i)\approx\sum_{j=0}^p\frac{a_j(x_i-x)^j}{j!}=\sc{a}{A(x_i-x)}_{\R^{p+1}},
\]
where $a=(a_0,\dots,a_p)^\top$, $A(y)=(1,y,\dots,y^p/p!)^\top$, the basic idea is to maximise the local polynomial log-likelihood
\begin{equation}
	\label{eq: localpolylike}
	l(\mu,x)=\sum_{i=1}^nw_i(x)\log\left(f(Y_i,\sc{a}{A(x_i-x)}_{\R^{p+1}})\right)
\end{equation}
over $a\in\R^{p+1}$ and with weight functions $w_i$, $i=1,\ldots,n$. 
The estimator for ${\mu}(x)$ is then given by $\hat{\mu}(x)=\hat{a}_0$. 
For a smoothing parameter $h$ (bandwidth), only observations within a given window $(x-h,x+h)$ are used, and each observation in \eqref{eq: localpolylike} is weighted by $w_i(x)=\mathcal{W}\left(\frac{x_i-x}{h}\right)$.
Often, $\mathcal{W}$ is chosen as a positive kernel function, but in principle it can be more general. 
This approach can be extended to the multivariate case (cf., amongst others, \cite{Aerts_1997,Ruppert_1994,Fan_1996}), and is also close to local polynomial regression (cf.~\cite{tsybakov_introduction_2008,Gyoerfi_2002}) as a generalisation of it. 
We will adapt the above method in the next section to construct a nonparametric estimator for $\theta(x)$.

\paragraph{The weighted augmented MLE}
The local observation processes $X_{\delta}$, $X^\nabla_{\delta}$ and $X^\Delta_\delta$ as introduced in \eqref{eq: locob} are no longer Markovian, as the time evolution at the point $x_k$, $k\in\{1,\ldots,N\}$, depends on the entire spatial structure of $X$. 
Therefore, a general Girsanov theorem for multivariate Itô processes, cf.~\cite[Section 7.6]{liptser_statistics_2001}, results in the modified log-likelihood
\begin{equation}	\label{eq: loglike}
	\norm{K}^{-1}_{L^2(\R^d)}\sum_{k=1}^{N}\bigg(\int_{0}^{T}\sc{X(t)}{A_\theta^\ast K_{\delta,x_k}}\diff X_{\delta,k}(t)-\frac{1}{2}\int_{0}^{T}\sc{X(t)}{A_{\theta}^\ast K_{\delta,x_k}}^2\diff t\bigg),
\end{equation}
provided the driving Brownian motions ${W}_k$ in \eqref{eq: process xdelta} are independent. 
For parametric $\theta$, it is straightforward to derive an estimate based on the observed processes by maximising \eqref{eq: loglike}, as shown in \cite{altmeyer_anisotrop2021}. 
In our set-up, we assume instead that we can approximate $\theta$ locally by a constant, i.e., for some $\gamma\in\R^d$, 
\[\sc{X(t)}{A^*_\theta K_{\delta,x_k}}\approx aX^\Delta_{\delta,k}(t)+\gamma^\top X^{\nabla}_{\delta,k}(t), \quad 0\le t\le T.
\]
Note that approximations by a polynomial of degree $p\geq1$ result in additional observations, which we do not have access to and which cannot be recovered by convolution and a finite difference scheme, see Remark \ref{rmk: higher order} below. 
Therefore, we restrict our investigations to the local constant approximation. 

Note further that we cannot use the local likelihood approach introduced before directly since $\theta$ is incorporated in $\sc{X(t)}{A_\theta^\ast K_{\delta,x_k}}$ via $\sc{X(t)}{\nabla\cdot\theta K_{\delta,x_k}}\neq \theta(x_k)^\top X^\nabla_{\delta,k}(t)$. 
Instead, we adapt the underlying idea by weighting the contribution of the $k$-th summand individually. 
Hence, we maximise  
\begin{equation*}
	\begin{split}
		l_\delta(\gamma,x)&=\norm{K}^{-1}_{L^2(\R^d)}\sum_{k=1}^{N}w_k(x)\left(\int_{0}^{T}\right.\left(aX^\Delta_{\delta,k}(t)-\gamma^\top X^\nabla_{\delta,k}(t)\right)\diff X_{\delta,k}(t)\\
		&\hspace*{11em}-\left.\frac{1}{2}\int_{0}^{T}\left(aX^\Delta_{\delta,k}(t)-\gamma^\top X^\nabla_{\delta,k}(t)\right)^2\diff t\right)
	\end{split}
\end{equation*}
over $\gamma\in\R^d$ to derive the \emph{weighted augmented MLE}, given by
\begin{equation}
	\label{eq: weighted est1}
	\hat{\theta}_\delta(x)=-(\mathcal{I}^x_{\delta})^{-1}\sum_{k=1}^{N}w_k(x)\left(\int_{0}^{T}X^\nabla_{\delta,k}(t)\diff X_{\delta,k}(t)-\int_{0}^{T}aX^\Delta_{\delta,k}(t)X^\nabla_{\delta,k}(t)\diff t\right)
\end{equation}
with the \emph{weighted observed Fisher information}
\begin{equation*}
	\label{eq: weighted fisher}
	\mathcal{I}^x_{\delta}=\sum_{k=1}^{N}w_k(x)\int_{0}^{T}X^\nabla_{\delta,k}(t)X^\nabla_{\delta,k}(t)^\top \diff t.
\end{equation*}

\begin{remark}[higher order approximations]\label{rmk: higher order}
Intuitively, approximating $\theta$ by a polynomial of higher order should perform an automatic bias correction and should thus improve the quality of the estimation. However, due to the spatial influence of $\theta$ in $\sc{X(t)}{A^*_\theta K_{\delta,x_k}}$, i.e., in 
	$$\sc{X(t)}{\nabla\cdot\theta K_{\delta,x_k}},$$
	the log-likelihood \eqref{eq: loglike} depends not only on the pointwise evaluations $\theta(x_k)$, $k=1,\ldots,N$, but rather on $\theta$ and $\nabla\cdot\theta$ in a neighbourhood around $x_k$. While the processes $X^\nabla_{\delta,k}$ and $X^\Delta_{\delta,k}$ can, in principle, be obtained by observing $X_{\delta,y}(t)$ in a neighbourhood of $x_k$, cf.~\cite{altmeyer_nonparametric_2020}, this fails to hold true for the additional observations required for higher order approximations. Even in the simplest case, i.e., a linear approximation of the form $\theta(y)=\gamma+\Gamma(y-x)$, one obtains
	\[
		\sc{X(t)}{\nabla\cdot(\gamma+\Gamma(\cdot-x))K_{\delta,x_k}}=\gamma^\top X^\nabla_{\delta, k}(t)
		+\sc{X(t)}{(\Gamma(\cdot-x))^\top \nabla K_{\delta,x_k}}+\operatorname{tr}(\Gamma)X_{\delta,k}(t).
	\]
	Since $K_{\delta,x_k}$ takes only non-zero values in a neighbourhood around $x_k$, we could instead approximate the non-observable term on the right hand side of the last display, while also ignoring the lower order perturbation, by
	$$\sc{X(t)}{\nabla\cdot(\gamma+\Gamma(\cdot-x))K_{\delta,x_k}}\approx \gamma^\top X^\nabla_{\delta, k}(t)+(\Gamma(x_k-x))^\top X^\nabla_{\delta,k}(t). $$
	Extending this idea to arbitrary polynomial approximations yields an estimate of $\theta$ and its partial derivatives at point $x.$ The analysis of this estimator, however, is nonstandard and seems to provide only limited, if any, improvement over $\hat{\theta}_\delta(x)$ in \eqref{eq: weighted est1} as its resulting bias component also depends on the approximation error within the accessible and inaccessible observations which restricts the usage of arbitrary Hölder regularity.
\end{remark}

\section{Convergence in probability: Upper bound results}\label{subs: pointwise}
Plugging \eqref{eq: process xdelta} into \eqref{eq: weighted est1} yields the decomposition 
\begin{equation}
	\label{eq: est2}
	\hat{\theta}_\delta(x)=\theta(x)+(\mathcal{I}^x_{\delta})^{-1}\mathcal{R}^x_{\delta}-(\mathcal{I}^x_{\delta})^{-1}\mathcal{M}^x_{\delta}\norm{K}_{L^2(\R^d)},
\end{equation}
where the martingale term and the remainder, respectively, are specified as
\begin{align*}
	\mathcal{M}^x_{\delta}&=\sum_{k=1}^{N}w_k(x)\int_{0}^{T}X^\nabla_{\delta,k}(t)\diff W_k(t),\\
	\mathcal{R}^x_{\delta}&=\sum_{k=1}^{N}w_k(x)\int_{0}^{T}X^\nabla_{\delta,k}(t)\sc{X(t)}{((\theta-\theta(x)) \cdot\nabla+\varphi_\theta)K_{\delta,x_k}}\diff t,
\end{align*}
with $\varphi_\theta\coloneqq\nabla\cdot\theta-c\in\mathcal{H}(\beta-1)$.
The following assumption gathers technical conditions required for our statistical analysis.
\begin{customass}{L}\,
	Assume that the following conditions are satisfied:
	\begin{enumerate}[label=(\roman*)]
		\label{ass: total}
		\item 
		The locations $x$ and $x_1,\dots,x_N$ belong to a fixed compact set $\mathcal{J}\subset\Lambda$, which is independent of the resolution $\delta$ and $N$. 
		There exists $\delta'>0$ such that $\mathrm{supp}(K_{\delta,x_k})\cap\mathrm{supp}(K_{\delta,x_l})=\emptyset$ for $k\neq l$, $k,l\leq N$, and all $\delta\leq\delta'$. 
		\item 
		There exists a compactly supported function $\bar{K}\in H^4(\R^d)$, which is either even or odd, such that $K=(-\Delta)\bar{K}$.
		\item 
		Given $h>0$, there exist weight functions $w_k=w_k(N,h,x_1,\dots,x_N)\colon\mathcal{J}\to\R$, depending only on $N,h,x_1,\dots,x_N$, fulfilling the following conditions for a universal constant $C_*$ not depending on $N$, $\delta$, $h$ and $x$: 
		\begin{enumerate}
			\item [(1)] $\sup_{k\in[N], x\in\mathcal J}|w_k(x)| \leq C_*(Nh^d)^{-1}$;
			\item [(2)] $\sum_{k=1}^{N}|w_k(x)|\leq C_*$;
			\item [(3)] $w_k(x)=0\text{ if }|x_k-x|>h$;
			\item [(4)] $\sum_{k=1}^{N}w_k(x)=1$, and for any $\alpha$ with $|\alpha|=1$, it holds $\sum_{k=1}^N(x_k-x)^\alpha w_k(x)=0$.
		\end{enumerate}
		\item The initial condition $X_0$ is such that either $X_0\in L^p(\Lambda)\cap \mathcal{D}(A_\theta)$, $p>2$, or, if in addition there exists a constant $\gamma<0$ such that $c-\nabla\cdot\theta\leq \gamma$, $X_0=\int_{-\infty}^0S_{\theta}(t')\diff W(t')$.
	\end{enumerate}
\end{customass}

A few comments on the above conditions are in order.
The support condition in Assumption \ref{ass: total}(i) guarantees that the Brownian motions $W_k$ in \eqref{eq: process xdelta} are independent as $\delta\rightarrow 0$, while the processes $X_{\delta,k}$ defined in \eqref{eq: locob} do not inherit independence. 
It requires the measurement locations $x_k$ to be separated by a Euclidean distance of at least $C\delta$ for a fixed constant $C$, which means that $N$ grows at most as $N = O(\delta^{-d})$. 
Existence of weight functions $w_k$ in Assumption \ref{ass: total}(iii) holds true under standard structural assumptions on the locations $x_k$ (see Lemma \ref{lem: existweight} and the subsequent remark below). 
Since the partial derivatives $\partial_iK$, $i=1,\ldots, d$, are mutually independent, condition (4) also implies that $\mathcal{I}_\delta^x$ is $\P$-a.s.~invertible, which can be deduced from \cite[Lemma 2.2]{altmeyer_anisotrop2021}. 
The weights can be constructed similarly to weights in local polynomial regression, cf.~for instance \cite[Chapter 1.6]{tsybakov_introduction_2008}, such that they are reproducing of order one. 
Assumption \ref{ass: total}(iv) guarantees that a general initial condition is asymptotically neglectable. 
If $\theta=0$, it can be further relaxed such that $\gamma=0$ is allowed, i.e., $X_0=\int^0_{-\infty}S_\theta(t')\diff W(t')$ is, for instance, also valid for $A_\theta=a\Delta$.
Despite using the local constant (LP$(0)$) approach, we will show that $\hat{\theta}_\delta(x)$ achieves the convergence rate of an LP$(1)$-estimator. 
While in local polynomial regression, this is known to happen for the Nadaraya--Watson estimator if, for instance, one works with equidistant design points $x_k$ and estimates at one of those locations, see Example \ref{ex: weights}, we only rely on a first-order multivariate Taylor expansion and use the reproducing property of the weights as well as (anti-) symmetry of $\nabla K$, implied by Assumption \ref{ass: total}(ii). 
Depending on more information about the initial condition and the dimension $d$, Assumption \ref{ass: total}(ii) can also be softened.

A precise control of the error decomposition \eqref{eq: est2} results in consistency of the estimate $\hat{\theta}_\delta(x)$ as the resolution level $\delta$ tends to zero. As known from the parametric case, cf.~\cite{altmeyer_anisotrop2021}, consistent estimation in finite time $T$ of the velocity $\theta$  naturally requires $N=N(\delta)\rightarrow\infty$ as $\delta\rightarrow0$. 
On the other hand, the bandwidth $h\rightarrow0$ is usually chosen in dependence on the number $N$ of observations to balance between bias and variance terms. In that case, $h$ is implicitly also related to $\delta$. 

\begin{theorem}
	\label{thm: maintheorem}
	Under Assumption \ref{ass: total}, the weighted augmented MLE satisfies 
	\begin{equation}
		\label{eq: mainresult}
		\hat{\theta}_\delta(x)-\theta(x)=O_{\P}(h^{\beta}+(Nh^d)^{-1/2}),\quad \beta\in(1,2].
	\end{equation}
 In particular, this bound is independent of the spatial location $x\in\mathcal{J}$ in the sense that, for any $\varepsilon>0$, there exist some $M>0$, $\delta'>0$ such that, for any $x\in\mathcal{J}$ and for any $\delta\leq\delta'$, we have
 \begin{equation}
     \label{eq: extraresult}
     \P\left(\big|\hat{\theta}_\delta(x)-\theta(x)\big|\big(h^\beta+(Nh^d)^{-1/2}\big)^{-1}>M\right)\leq\varepsilon.
 \end{equation}
\end{theorem}

To achieve consistency in the first place, the above result implies that $Nh^d\to\infty$ is required.
Hence, it can only hold if $\delta\ll h$, since Assumption \ref{ass: total}(i) imposes at most $N\asymp\delta^{-d}$ measurement locations.

\begin{remark}[convergence rate]
	Optimising the upper bound stated in \eqref{eq: mainresult} with respect to the bandwidth $h$ yields
	\begin{equation}
		\label{eq: orderh1}
		h\asymp N^{-1/(2\beta+d)},\quad \text{ that is, }\quad h^\beta\asymp N^{-\beta/(2\beta+d)},\quad \beta\in(1,2],
	\end{equation}
	thus matching the standard rates for the mean-squared error in nonparametric regression. The usual bias-variance trade-off, resulting from choosing suboptimal $h$, is illustrated in Figure \ref{fig: sim}. 
	For a maximal choice $N\asymp\delta^{-d}$, the optimal bandwidth specification gives
	\begin{equation}
		\label{eq: orderh2}
		h\asymp \delta^{d/(2\beta+d)},\quad \text{ that is, }\quad h^\beta\asymp\delta^{\beta d/(2\beta+d)},\quad \beta\in(1,2].
	\end{equation}
	A graphical illustration in $d=1$ for $\beta=2$, i.e., $h^{\beta}\asymp\delta^{2/5}$, is given in Figure \ref{fig: sim}. As demonstrated in Section \ref{sec: optimality}, the rates in \eqref{eq: orderh1} and \eqref{eq: orderh2} are optimal. 
\end{remark}
Naturally, one may ask if the rate in \eqref{eq: orderh1} also holds true under higher order Hölder regularity assumptions. Indeed, Theorem \ref{thm: maintheorem} might, in principle, be extended to arbitrary $\beta>2$, using reproducing weights functions $w_k$ of order $\lfloor\beta\rfloor$ instead. The analysis of the remainder $\mathcal{R}_\delta^x$ in Section \ref{sec: restterm}, however, indicates that its order is not determined by the bandwidth $h$ and smoothness parameter $\beta$ alone, yet also dependent on the resolution level $\delta.$ In particular, 
$$\mathcal{R}_\delta^x=O_{\P}(h^\beta+\delta h+\delta^2).$$
Thus, the dominating term varies, depending on the dimension $d$ and the assumed smoothness $\beta$. 
If $\beta\leq2$, the remainder is always of order $h^\beta$ whilst the parametric order $\delta^2$ can be achieved for $d\rightarrow\infty$ and arbitrary $\beta\geq2$. 
This matches the observations made in \cite{altmeyer_nonparametric_2020,altmeyer_anisotrop2021}, where the bias term does neither depend on the time horizon $T$, the diffusivity $a$, nor the number of spatial observations $N$. 
As a consequence, arbitrary $\beta>1$ allow for the dimension-improving convergence rates $\delta^{\beta d/(2\beta+d)}\vee \delta^{2}$. 
This phenomenon, however, is no contradiction to the curse of dimensionality stated in \eqref{eq: orderh1} as it results by reparametrisation of $N$. 
Nonetheless, it is in harmony with the central limit theorem \cite[Theorem 2.3]{altmeyer_anisotrop2021} which also yields a better rate if $N$ is chosen maximal.
\begin{figure}[h!]
	\centering
	\begin{subfigure}{.49\textwidth}
		\centering\includegraphics[width=1.\linewidth]{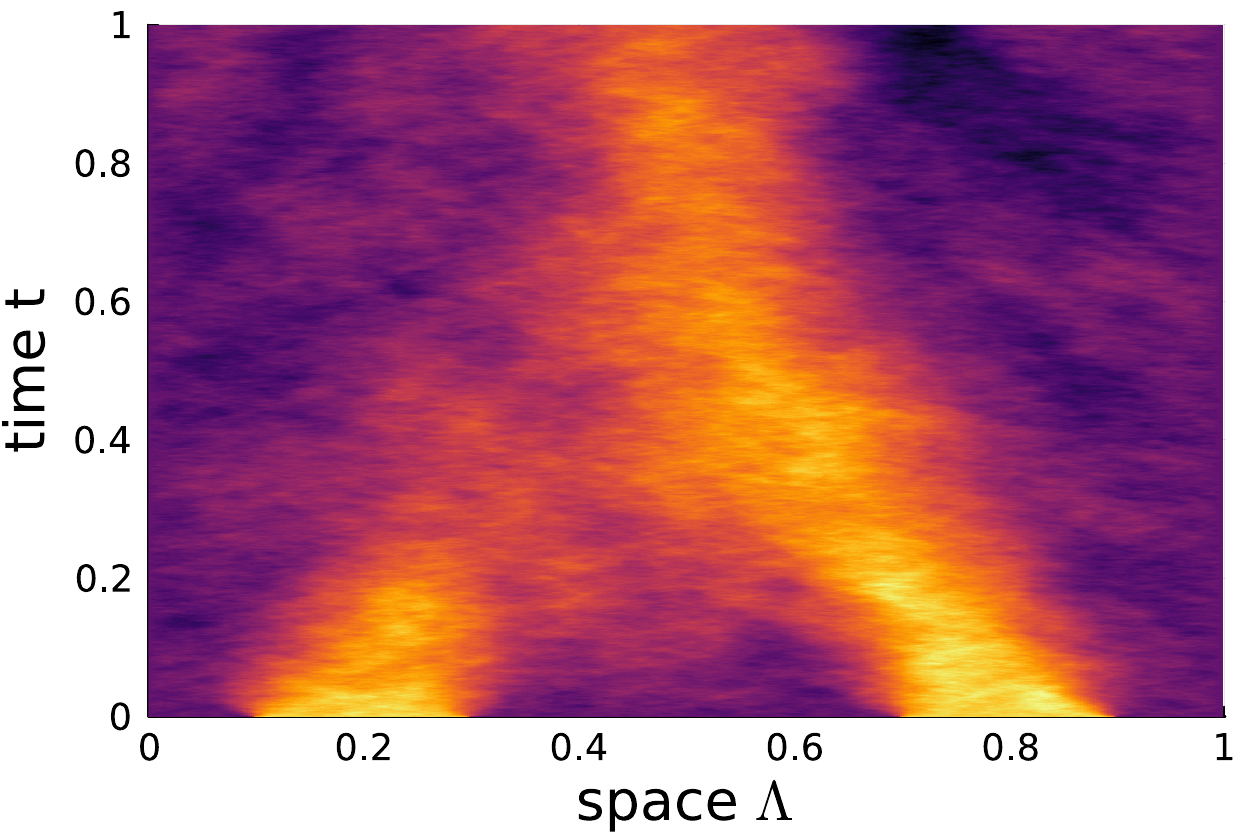}
	\end{subfigure}
	\begin{subfigure}{.49\textwidth}
		\centering\includegraphics[width=1.\linewidth]{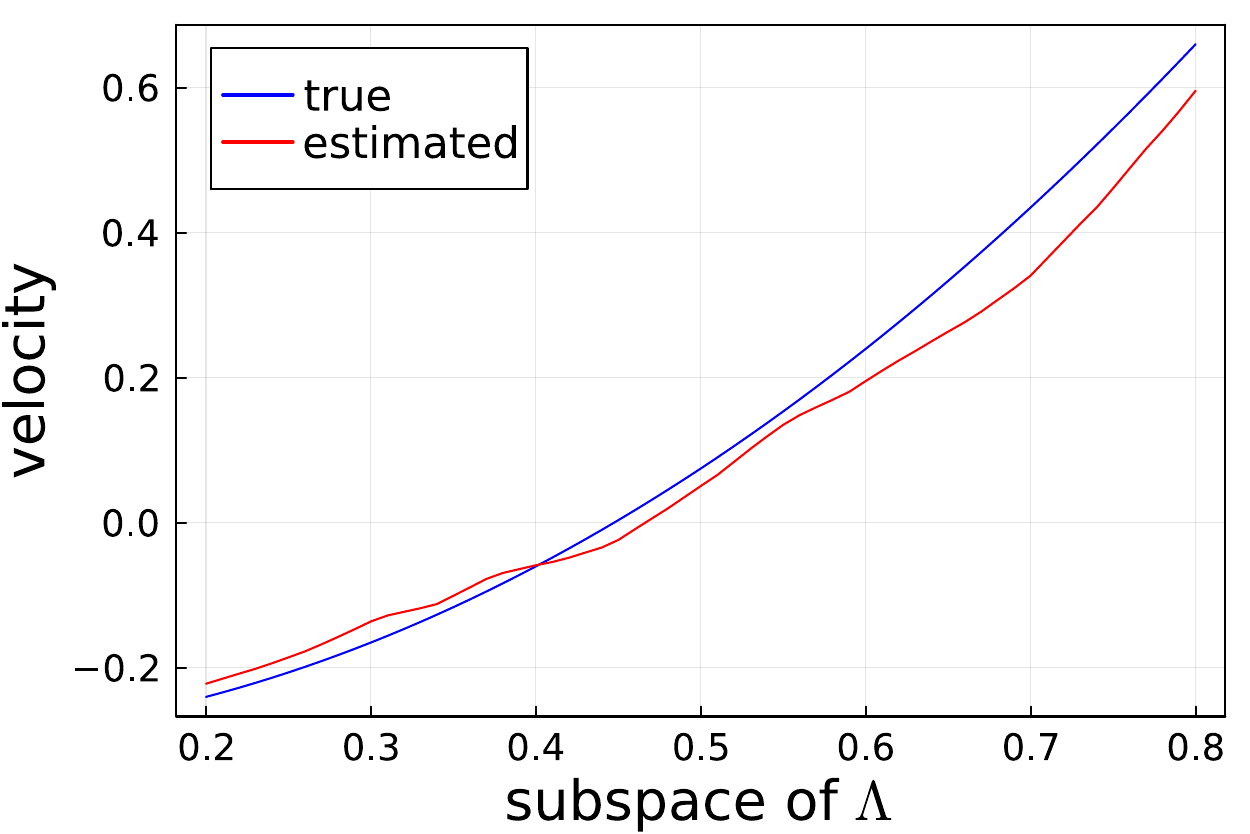}
	\end{subfigure}
	\begin{subfigure}{.49\textwidth}
		\centering\includegraphics[width=1.\linewidth]{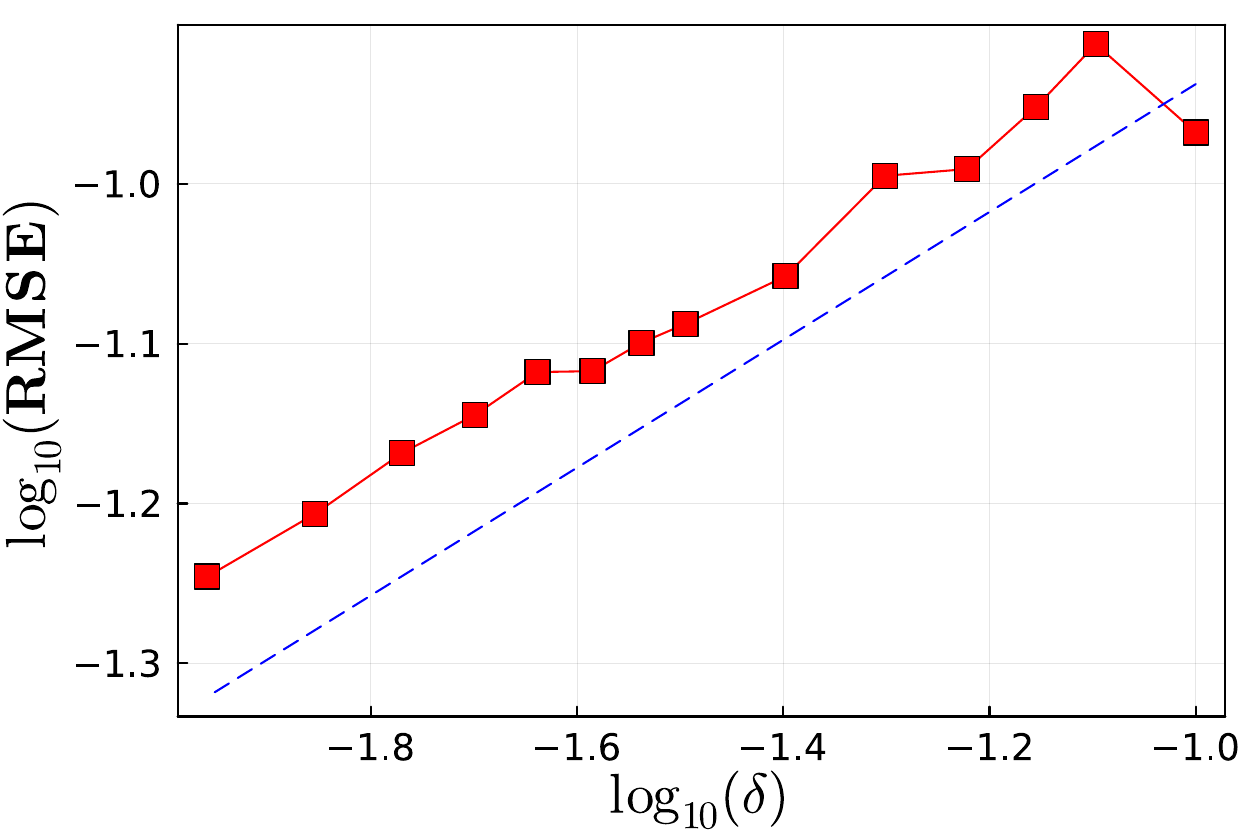}
	\end{subfigure}
	\begin{subfigure}{.49\textwidth}
		\centering\includegraphics[width=1.\linewidth]{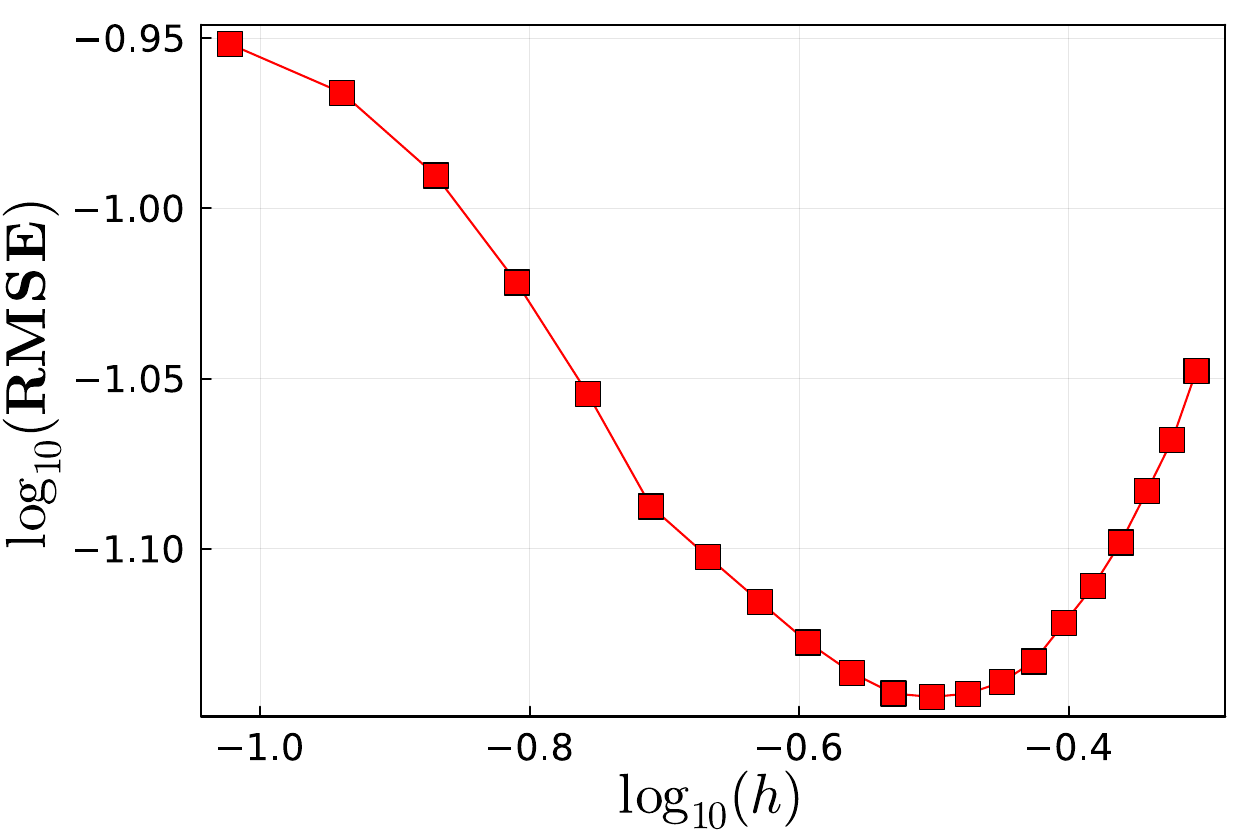}
	\end{subfigure}
	\caption{(top-left) typical realisation of the solution $X(t,x)$ in $d=1$ with domain $\Lambda=(0,1)$; (top-right) trajectory of $\hat{\theta}_\delta(x)$ compared to $\theta(x)=-0.3+1.5x^2$ in the interval $[0.2,0.8]\subset\Lambda$ with weights $w_k(x)$ based on the Epanechnikov kernel; (bottom) $\log$-$\log$ plot of the root mean squared error for estimating $\theta$ at $x=0.5$ with $\delta\rightarrow0$, $h\asymp \delta^{2/5}$ (left); $\delta$ fix, $h\rightarrow0$ (right).}
	\label{fig: sim}
\end{figure}

A second extension of Theorem \ref{thm: maintheorem} involves the diffusivity $a$. 
While the estimator $\hat{\theta}_\delta(x)$ in \eqref{eq: weighted est1} requires knowledge of this parameter, in general it may be unknown. Replacing thus $a$ by a reasonable estimate $\hat{a}_\delta$ yields another estimator $\tilde{\theta}_\delta(x)$ which achieves the same convergence rate.

\begin{corollary}
	\label{cor: unknown a}
	Grant Assumption \ref{ass: total}, and suppose that the diffusivity $a>0$ is unknown. 
	Define the estimator $\tilde{\theta}_\delta(x)$ similarly to \eqref{eq: weighted est1}, replacing $a$ with an estimate $\hat{a}_\delta$. 
	If $\hat{a}_\delta$ satisfies
	\begin{equation}
		\label{eq: orda}
		\hat{a}_\delta-a\in O_{\P}(h^\beta+(Nh^d)^{-1/2}),
	\end{equation}
	then $\tilde{\theta}_\delta(x)-\theta(x)\in O_{\P}(h^\beta+(Nh^d)^{-1/2})$.
\end{corollary}
Estimators which fulfill \eqref{eq: orda} are, for instance, given by
\begin{equation}
	\label{eq: esta}
	\hat{a}_\delta=\frac{\sum_{k=1}^Nw_k(x)\int_0^T X^\Delta_{\delta,k}(t)\diff X_{\delta,k}(t)}{\sum_{k=1}^Nw_k(x)\int_0^T X^\Delta_{\delta,k}(t)^2\diff t}\quad\text{or}\quad \hat{a}_\delta=\frac{\sum_{k=1}^N\int_0^T X^\Delta_{\delta,k}(t)\diff X_{\delta,k}(t)}{\sum_{k=1}^N\int_0^T X^\Delta_{\delta,k}(t)^2\diff t}.
\end{equation}

Finally, we can also extend Theorem \ref{thm: maintheorem} beyond the pointwise risk and quantify the quality of $\hat{\theta}_\delta$ on the whole domain $\Lambda$. 
Since the estimator $\hat{\theta}_\delta(x)$ in \eqref{eq: weighted est1} was only defined for $x\in\mathcal{J}$, we start by expanding its definition to $\Lambda$. 
Its value at $x\in\Lambda\setminus\mathcal{J}$ is set to a value $\hat{\theta}_\delta(x_0)$, whereas $x_0\in\mathcal{J}$ is closest to $x$, that is, 
\begin{equation}
	\label{eq: estwhole}
	\hat{\theta}_\delta(x)\coloneqq \inf_{x_0}\hat{\theta}_\delta(x_0),
\end{equation}
with $x_0\in\{x\in\mathcal{J}:|x-x_0|=\textrm{dist}(x,\mathcal{J})\}$.
Hence, we take the estimate at the closest point $x_0\in\mathcal{J}$ to further exploit Hölder continuity. 
The infimum over all possible $x_0$ is taken to obtain a unique estimate. 
Alternatively, one could also consider polynomial interpolation outside of $\mathcal{J}$.

\begin{corollary}
	\label{cor: intrisk}
	Grant Assumption \ref{ass: total}, and define $\hat{\theta}_\delta$ outside of $\mathcal{J}$ via \eqref{eq: estwhole}. 
	Then,
	\begin{equation}
		\label{eq: intrisk}
		\int_\Lambda\left(\hat{\theta}_\delta(x)-\theta(x)\right)^2\diff x=O_{\P}\left(h^{2\beta}+\frac{1}{Nh^d}\right)+O(d^2_{\max}\lambda(\Lambda\setminus\mathcal{J})),
	\end{equation}
	where $\lambda$ denotes the Lebesgue measure on $\R^d$ and $d^2_{\max}\coloneqq \sup_{x\in\Lambda\setminus\mathcal{J}}\operatorname{dist}^2(x,\mathcal{J})$ is the maximal squared distance of $\mathcal{J}$ to the boundary $\partial\Lambda$. 
\end{corollary}

\begin{remark}[discussion of Corollary \ref{cor: intrisk}]
	Equation \eqref{eq: intrisk} splits the squared integrated error into a term of stochastic order, similar to the pointwise risk in Theorem \ref{thm: maintheorem}, and a deterministic part which is entirely dependent on the compact set $\mathcal{J}\subset\Lambda$. 
	While the question of consistency thus is not immediately clear, it still can be achieved with a (possibly) slower rate. 
	The supports of $K_{\delta,x}$ are contained in $\bar{\Lambda}$ for all $x\in\mathcal{J}$ and any $\delta\leq\delta'$ for $\delta'$ small enough due to the compactness of $\mathcal{J}$. 
	This means that the distance between the boundary $\partial\Lambda$ and $\mathcal{J}$ behaves at best like $\delta'$, i.e., $d^2_{\max}\asymp(\delta')^2$. 
	On the other hand, $\lambda(\Lambda\setminus\mathcal{J})$ becomes small if $d^2_{\max}$ decreases. 
	In fact, $d^2_{\max}\asymp(\delta')^2$ implies $\lambda(\Lambda\setminus\mathcal{J})=O(\delta')$. Hence, under a maximal choice of $N$ and optimisation in $h$, \eqref{eq: intrisk} yields the order
	$$O_{\P}(\delta^{2\beta d/(2\beta+d)})+O((\delta')^3).$$
	Cases where $d^2_{\max}\asymp\delta'^2$ are given, for instance, if 
	\begin{itemize}
		\item $\Lambda$ is an $d$-dimensional open ball of radius $r$, and $\mathcal{J}$ is the closed ball with radius $r-\delta'$ and the same centre point;
		\item $\Lambda$ is a rectangular cuboid of the form $(a_1,b_1)\times\dots\times(a_d,b_d)$, and $\mathcal{J}$ is chosen as $[a_1+\delta',b_1-\delta']\times\dots\times[a_d+\delta',b_d-\delta']$.
	\end{itemize}
\end{remark}

Let us finish this section with a closer inspection of the weight functions $w_k(x)$ from Assumption \ref{ass: total}(iii). 
Their existence holds under general design assumptions, cf.~also \cite[Lemma 1.4 and Lemma 1.5]{tsybakov_introduction_2008}.

\begin{lemma}
	\label{lem: existweight}
	Let $h>0$ and $V\colon\R^d\rightarrow\R$ be a kernel function. 
	Consider the $\R^{d+1}$-valued function $U$ given by $U(u)=(1,u_1,\dots,u_d)^\top$, and define the matrix 
	$$B_{Nx}=\frac{1}{Nh^d}\sum_{k=1}^NU\left(\frac{x_k-x}{h}\right)U^\top\left(\frac{x_k-x}{h}\right)V\left(\frac{x_k-x}{h}\right).$$
	Assume that the following conditions hold:
	\begin{itemize}
		\item[$\operatorname{(LP1)}$] There exist a real number $\lambda_0>0$ and a positive integer $n_0$ such that the smallest eigenvalue $\lambda_{\min}(B_{Nx})\geq\lambda_0$ for all $n\geq n_0$ and any $x\in\mathcal{J}$.
		\item[$\operatorname{(LP2)}$] 
		There exists a real number $a_0>0$ such that, for any $A\subset\mathcal{J}$ and all $N\geq1$,
		\[
		\frac{1}{N}\sum_{k=1}^N\mathbf{1}(x_k\in A)\leq a_0\max(\lambda(A),1/N),
		\]
		with $\lambda$ denoting the Lebesgue measure.
		\item[$\operatorname{(LP3)}$] 
		The kernel $V$ has compact support in $[-1,1]^d$, and there exists a number $V_{\max}<\infty$ such that $V(u)\leq V_{\max}$ for all $u\in\R^d$.
	\end{itemize} 
	Then, the weights defined by
	\begin{equation}
		\label{eq: weightfunction}
		w_k(x)\coloneqq\frac{1}{Nh^d}U^\top(0)B_{Nx}^{-1}U\left(\frac{x_k-x}{h}\right)V\left(\frac{x_k-x}{h}\right)
	\end{equation}
	satisfy Assumption \ref{ass: total}(iii).
\end{lemma}

Assumptions (LP1)-(LP3) in Lemma \ref{lem: existweight} are satisfied under reasonable constraints on the design points $x_1,\dots,x_N$. 
(LP2) means that they are densely enough distributed over $\mathcal{J}$. 
This holds true, for instance, under equidistant design, noting that at most $N\asymp\delta^{-d}$. 
(LP1) is satisfied if $V(u)>V_{\min}>0$ in a neighbourhood around 0 and if additionally $x_1,\dots,x_N$ are sufficiently dense in $\mathcal{J}$, cf.~\cite[Lemma 1.4 and Lemma 1.5]{tsybakov_introduction_2008}. 
(LP3) presents no restriction since the kernel $V$ can be chosen according to need.

\begin{example}
	\label{ex: weights}
	Let us give a concrete example of the weights $w_k$ in \eqref{eq: weightfunction}.
	Assume $d=1$, $\mathcal{J}=[0,1]$, and choose the rectangular kernel $V(y)=\mathbf{1}(-1/2\leq y\leq 1/2)$. 
	Define $I_h=\{k:|x_k-x|\leq h/2\}$. 
	Then, 
	\[
	B_{Nx}=\frac{1}{Nh}\begin{pmatrix}\sum_{k\in I_h}1&\sum_{k\in I_h}\frac{x_k-x}{h}\\\sum_{k\in I_h}\frac{x_k-x}{h}&\sum_{k\in I_h}\left(\frac{x_k-x}{h}\right)^2
	\end{pmatrix}
	\]
	has strictly positive determinant if there are at least two different points $x_i,x_j$ in an $h/2$-neighbourhood around $x$. 
	If the measurement points $x_k$ are equidistantly distributed on $\mathcal{J}$, that is, if $x_k=(k-1)/(N-1)$, $k=1,\dots,N$, and we estimate at the location $x=r/(N-1)$, $1\leq r\leq N-2$, with $h<\min(x,1-x)/2$, then $\sum_{k\in I_h}(x_k-x)/h=0$ by symmetry. 
	The weights $w_k(x)$ in \eqref{eq: weightfunction} are given by
	$$w_k(x)=(\#I_h)^{-1}\mathbf{1}(k\in I_h).$$
	In that case, the weights correspond to the weight function of the Nadaraya--Watson estimator with rectangular kernel.
	
	An estimated trajectory based on the weights in \eqref{eq: weightfunction} with Epanechnikov kernel $V(y)=0.75(1-y^2)\mathbf{1}(|y|\leq1)$ is given in Figure \ref{fig: sim}.
\end{example}

\section{Lower bounds}\label{sec: optimality}
The convergence rate $N^{-\beta/(2\beta+d)}$ established for the weighted augmented MLE in Theorem \ref{thm: maintheorem} is optimal and cannot be improved in our general setup, as will be shown in this section.
We will only consider submodels $\P_\theta$ such that $A_\theta$ involves a negative reaction term, assuming a sufficiently regular kernel function $K$ and a stationary initial condition.

\begin{customass}{O}
	\label{ass:lowerBound} Suppose that $\P_{\theta}$ corresponds to the law of the stationary solution $X$ to the SPDE \eqref{eq: SPDE}, and assume that the following conditions hold:
	\begin{enumerate}[label=(\roman*)]
		\item The kernel function satisfies $K=\Delta^2\tilde{K}$ with $\tilde{K}\in C_c^{\infty}(\R^d)$.
		\item The model is $A_{\theta}=\Delta+\theta\cdot\nabla+c$ with a nonpositive reaction function $c\colon \Lambda\rightarrow\R$ and such that $\theta\colon \Lambda\rightarrow\R$ lies in the class $\Theta$ of $\beta$-Hölder continuous functions with the properties that there exists a constant $\gamma\leq0$ such that the $(\beta-1)$ Hölder-continuous function $c-\nabla\cdot\theta$ is smaller or equal than $\gamma$ and that $\theta$ is a conservative vector field.
		\item Let $x_1,\dots,x_N$ be $\delta$-separated points in $\Lambda$, that is, $|x_k-x_l|> \delta$ for all $1\leq k\neq l\leq N$. Moreover, suppose that $\mathrm{supp}(K_{\delta,x_k})\subset\Lambda$ for all $k=1,\dots,N$, and that $\mathrm{supp}(K_{\delta,x_k})\cap\mathrm{supp}(K_{\delta,x_l})=\emptyset$ for all $1\leq k\neq l\leq N$.
	\end{enumerate}
\end{customass}
We consider the null model to be $A_\theta=\Delta$, i.e., $\theta=0$, $c=\gamma=0$, and we test against alternatives where $\theta\neq0$ and $c$ is strictly negative such that $c-\nabla\cdot\theta\leq\gamma<0$.

\begin{theorem}\label{thm:lower:bound:M>1}
	Grant Assumption \ref{ass:lowerBound}. 
	Then, there exist $c_1>0$, depending only on $K$ and $d$, and an absolute constant $c_2>0$ such that, for any $x\in\Lambda$, the following assertion holds:
	\begin{align*}
		\inf_{\hat\vartheta}\sup_{\vartheta\in\Theta}
		\P_\vartheta\left(|\hat\vartheta(x)-\vartheta(x)|\geq \frac{c_1}{2}N^{-\beta/(2\beta+d)}\right)>c_2,
	\end{align*}
	where the infimum is taken over all real-valued estimators $\hat\vartheta_i=\hat\vartheta_i(X_{\delta})$.
\end{theorem}
As the weighted augmented MLE is not only based on the observations of $X_\delta$, but also on $X^\Delta_\delta$ and $X^\nabla_\delta$, Theorem \ref{thm:lower:bound:M>1} can be furthermore extended to estimators $\hat{\theta}$ using those additional observations. 

\begin{theorem}\label{thm:lower:bound:M>1:add:measurements}
	Theorem \ref{thm:lower:bound:M>1} remains valid when the infimum is taken over all real-valued estimators $\hat\vartheta_i=\hat\vartheta_i(X_{\delta},X_{\delta}^\Delta,X_{\delta}^{\nabla})$, provided that $K$, $\Delta K$ and $\partial_i K$ are independent and Assumption \ref{ass:lowerBound}(i) holds for  $K$, $\Delta K$ and $\partial_i K$, $1\leq i\leq d$.
\end{theorem}
Theorem \ref{thm:lower:bound:M>1} is proven in Section \ref{sec:prooflower} below. The proof of Theorem \ref{thm:lower:bound:M>1:add:measurements} is skipped as it relies only on minor modifications, see also \cite[Theorem 3.3]{altmeyer_anisotrop2021}.

\section{Technical supplement: Auxiliary results and proofs}\label{sec: proofs}
We start with a few initial notations and remarks. 
Write $\Lambda_{\delta,y}=\{\delta^{-1}(u-y)\colon u\in\Lambda\}$, $\Lambda_{0,y}=\R^d$, and introduce the rescaled operators $A_{\theta,\delta,y}$ and $\bar A_{\delta,y}$ with domain $H^1_0(\Lambda_{\delta,y})\cap H^2(\Lambda_{\delta,y})$ by setting
\[
A_{\theta,\delta,y}\coloneqq a\Delta +\delta\theta(y+\delta\cdot)\cdot\nabla+\delta^2c(y+\delta\cdot),\quad 
\bar{A}_{\delta,y}\coloneqq a\Delta.
\]
The associated analytic semigroups on $L^2(\Lambda_{\delta,y})$ are denoted by $(S_{\theta,\delta,y}(t))_{t\geq0}$ and $(\bar{S}_{\delta,y}(t))_{t\geq0}$, respectively. 
Write $\e^{ta\Delta}$ for the semigroup on $L^2(\R^d)$ generated by $a\Delta$ on $H^2(\R^d)$.
Define the heat kernel $q_t(u)=(4\pi t)^{-d/2}\exp(-|u|^2/(4t))$, and notice that, for $(\e^{ta\Delta})z=q_{at}\ast z$, by Young's inequality, 
\[
\norm{\e^{ta\Delta}z}_{L^2(\R^d)}\lesssim (1\wedge t^{-d/4})(\norm{z}_{L^1(\R^d)}+\norm{z}_{L^2(\R^d)}).
\]
We denote $\varphi_\theta=\nabla\cdot\theta-c$, and we want to estimate $\theta$ at the (fixed) location $x\in\mathcal{J}$.
The stochastic order $O_{\P}(h^\beta+\delta h+\delta^2)$ of $\mathcal{R}_\delta^x$, which can, in principle, be found from the proofs in Section \ref{sec: restterm} below, will always be dominated by $h^\beta$ as $\beta\leq2$. 
This is clear since our methodology is only applicable if $\delta\ll h$. 
Indeed, if $h\leq\delta$, then consistency cannot be achieved as the number of observations used to construct the estimator in \eqref{eq: weighted est1} remains finite. Optimising \eqref{eq: mainresult} with respect to the bandwidth $h$ yields that $h\asymp N^{-1/(2\beta+d)}$. 
Furthermore, Assumption \ref{ass: total} implies that there exist at most $N\asymp \delta^{-d}$ spatial observation locations. 
Together, this gives for any dimension $d\geq1$ and $\beta\in(1,2]$,
\begin{equation}
	\label{eq: order h}
	\delta^2\ll \delta^{\beta d/(2\beta+d)}=O(h^\beta),\quad \delta^{d/2}\ll h^\beta,
\end{equation}
which we will frequently use in Sections \ref{sec: properties measurements} and \ref{sec: detailederror} down below. 

\subsection{The rescaled semigroup}
\label{sec: rescaledsemigroup}
In this section, we present properties of the rescaled semigroup $(S^*_{\theta,\delta,y}(t))_{t\geq0}$ and its infinitesimal generator $A^*_{\theta,\delta,y}.$

\begin{lemma}[Lemma 3.1 of \cite{altmeyer_nonparametric_2020}]
	\label{rescaledsemigroup}
	For $\delta>0$ and $y\in\Lambda$, it holds:
	\begin{itemize}
		\item[$\operatorname{(i)}$] If $z\in H^1_0(\Lambda_{\delta,y})\cap H^2(\Lambda_{\delta,y})$, then $A^\ast_{\vartheta}z_{\delta,y}=\delta^{-2}(A^\ast_{\vartheta,\delta,y}z)_{\delta,y}$;
		\item[$\operatorname{(ii)}$] if $z\in L^2(\Lambda_{\delta,y})$, then $S^\ast_{\vartheta}(t)z_{\delta,y}=(S^\ast_{\vartheta,\delta,y}(t\delta^{-2})z)_{\delta,y}$, $t\geq0$.
	\end{itemize}
\end{lemma}

The following lemma is a classical result for sectorial operators and corresponding analytic semigroups. 
Our version holds for growing domains $\Lambda_{\delta,y}$, uniformly in $y\in\mathcal{J}$.

\begin{lemma}\label{boundS*_01}
	There exist universal constants $M_0,M_1,C>0$ such that, for $\delta\geq 0$, $t>0$, 
	\begin{align*}
		\sup_{y\in \mathcal{J}}\norm{S^\ast_{\vartheta,\delta,y}(t)}_{L^2(\Lambda_{\delta,y})}&\le M_0\e^{C\delta^2t},\\
		\sup_{y\in \mathcal{J}}\norm{t(C\delta^2I-A^\ast_{\theta,\delta,y})S_{\vartheta,\delta,y}^\ast(t)}_{L^2(\Lambda_{\delta,y})}&\le M_1\e^{C\delta^2t}.
	\end{align*}
\end{lemma}
This lemma shows that the shifted semigroup $\e^{-2C\delta^2t}S^\ast_{\theta,\delta,y}(t)$ decays exponentially,  
\[
\norm{\e^{-2C\delta^2t}S^\ast_{\theta,\delta,y}(t)}_{L^2(\Lambda_{\delta,y})}\le\e^{-C\delta^2t},
\] 
and so the resolvent set of the correspondingly shifted infinitesimal generator $2C\delta^2 - A^\ast_{\theta,\delta,y}$ contains the right half of the complex plane. 
This allows for defining the fractional powers $(2C\delta^2 - A^\ast_{\theta,\delta,y})^{s}$ for $s\in\R$, see \cite[Section 4.4]{hairer_introduction_2009}, and we obtain by \cite[Proposition 4.37]{hairer_introduction_2009} the usual smoothing property of analytic semigroups.

\begin{lemma}\label{boundS*_s}
	There exists a universal constant $M_2$ such that, for $\delta\geq 0$, $t>0$ and $s\geq0$,
	\[
	\sup_{y\in \mathcal{J}}\norm{ t^{s}(2C\delta^2-A^\ast_{\vartheta,\delta,y})^{s}S^\ast_{\vartheta,\delta,y}(t)}_{L^2(\Lambda_{\delta,y})}\leq M_2\e^{C\delta^2t}.
	\]
\end{lemma}
Intuitively, letting $\delta\to0$, the semigroup on $\Lambda_{\delta,y}$ will be close to the semigroup on $\R^d$. 
The following auxiliary result states this more precisely. 

\begin{lemma}
	\label{FeynmanKac}
	Let $t>0$, and grant Assumption \ref{ass: total}.
	\begin{itemize}
		\item[$\operatorname{(i)}$] There exist universal constants $c_1,c_2,c_3$ such that, if $z\in C_c(\mathbb{R}^d)$ is supported in $\bigcap_{y\in \mathcal{J}}\Lambda_{\delta,y}$ for some $\delta\geq 0$, then
		\[
		\sup_{y\in \mathcal{J}}\left|(S^\ast_{\vartheta,\delta,y}(t)z)(u)\right|\leq c_3\e^{c_1\delta^2t}(q_{c_2t}\ast|z|)(u),\quad u\in\R^d.
		\]
		\item[$\operatorname{(ii)}$] 
		If $z\in L^2(\mathbb{R}^d)$, then, as $\delta \to 0$,
		\[
		\sup_{y\in \mathcal{J}}\norm{ S^\ast_{\vartheta,\delta,y}(t)(z|_{\Lambda_{\delta,y}})-\e^{ta\Delta}z}_{L^2(\mathbb{R}^d)}\to0.
		\]
		\item[$\operatorname{(iii)}$] 
		If $z\in L^2(\mathbb{R}^d)$, then, for any $t\geq0$,
		\begin{equation*}
			\sup_{y\in \mathcal{J}}\norm{\bar{S}_{\delta,y}(t)z-\e^{ta\Delta}z}_{L^2(\R^d)}\lesssim \delta^{1/2}t^{1/4}\e^{-\delta^{-2}t^{-1}/2}.
		\end{equation*} 
	\end{itemize}
\end{lemma}
The action of the semigroup operators $S_{\theta,\delta,y}^\ast(t)$ applied to functions of a certain smoothness and integrability is given in the next lemma. 
The proof relies on the Bessel potential spaces $H_0^{s,p}(\Lambda_{\delta,y})$, $1<p<\infty$, $s\in\R$, defined for $\delta>0$ as the domains of the fractional weighted Dirichlet--Laplacian $(-\bar{A}_{\theta,\delta,y})^{s/2}$ of order $s/2$ on $\Lambda_{\delta,y}$ with norms $\norm{\ \cdot\ }_{H^{s,p}(\Lambda_{\delta,y})}=\norm{(-\bar{A}_{\theta,\delta,y})^{s/2}\cdot}_{L^p(\Lambda_{\delta,y})}$.

\begin{lemma}[Lemma 6.4 of \cite{altmeyer_anisotrop2021}]	\label{boundS*u}
	Let $\delta\in[0,1]$, $t>0$, and grant Assumption \ref{ass: total}. 
	Let $z\in H^{s}_0(\mathbb{R}^d)$, $s\geq0$, be compactly supported in $\bigcap_{y\in \mathcal{J}}\Lambda_{\delta,y}$, and suppose that $V_{\delta,y}\colon L^p(\Lambda_{\delta,y})\to H_0^{-s,p}(\Lambda_{\delta,y})$ are bounded linear operators with $\norm{V_{\delta,y}z}_{H^{-s,p}(\Lambda_{\delta,y})}\leq V_{\operatorname{op}}\norm{z}_{L^p(\Lambda_{\delta,y})}$, for some $V_{\operatorname{op}}$ independent of $\delta$ and $y$. 
	Then, there exists a universal constant $C>0$ such that, for $1< p\leq2$ and $\gamma=(1/p-1/2)d/2$,
	\begin{align*}
		&\sup_{y\in \mathcal{J}}\norm{ S_{\vartheta,\delta,y}^\ast(t) V_{\delta,y}z }_{L^2(\Lambda_{\delta,y})}\leq C\e^{c_1t\delta^2}\sup_{y\in \mathcal{J}}\left(\norm{V_{\delta,y}z}_{L^2(\Lambda_{\delta,y})}\wedge (V_{\operatorname{op}}t^{-s/2-\gamma}\norm{z}_{L^{p}(\Lambda_{\delta,y})})\right),
	\end{align*}
	where  $c_1$ is the constant described in Lemma \ref{FeynmanKac}$\operatorname{(i)}$.
	If $s=0$, the inequality holds also for $p=1$.
\end{lemma}

\subsection{Properties of multiple local measurements}\label{sec: properties measurements}
For the reader's convenience, we give the result of \cite{altmeyer_anisotrop2021}, specifying the covariance function of the Gaussian process defined in \eqref{eq: weaksol}.

\begin{lemma}[Lemma 6.5 of \cite{altmeyer_anisotrop2021}]\label{lem:covFun}
	\begin{enumerate}
		\item[$\operatorname{(i)}$] 
		If $X_0=0$, then the Gaussian process from \eqref{eq: weaksol} has mean zero and covariance function
		\[
		\operatorname{Cov}(\sc{X(t)}{z},\sc{X(t')}{z'}) = \int_0^{t\wedge t'} \sc{ S^\ast_{\theta}(t-s)z}{ S^\ast_{\theta}(t'-s)z'}\diff s.
		\]
		\item[$\operatorname{(ii)}$] 
		If $X_0$ is the stationary initial condition from Assumption \ref{ass: total}(iv), then the Gaussian process from \eqref{eq: weaksol} has mean zero and covariance function
		\begin{align*}
			&\operatorname{Cov}(\sc{X(t)}{z},\sc{X(t')}{z'}) = \int_0^{\infty} \sc{ S^\ast_{\theta}(t+s)z}{ S^\ast_{\theta}(t'+s)z'}\diff s.
		\end{align*}
	\end{enumerate}
\end{lemma}

\begin{lemma}
	\label{lem: ConvFisher}
	Grant Assumption \ref{ass: total}, and consider $u,w\in\{D^\alpha K:|\alpha|\leq2\}=\{-D^\alpha \Delta\bar{K}:|\alpha|\leq2\}$. 
	Let $X_0=0$, and set $f_0(t)\coloneqq \sc{\e^{ta\Delta}u}{\e^{ta\Delta}w}_{L^2(\R^d)}$.
	Then, the following properties hold true:
	\begin{enumerate}
		\item[$\operatorname{(i)}$] $\psi({u},{w}) = \int_0^{\infty}f_0(t)\diff t$ is well-defined, i.e., $f_0\in L^1([0,\infty))$.
		\item[$\operatorname{(ii)}$] For $\delta\to0$, 
		\[
		\sup_{y\in \mathcal{J}}\left|\delta^{-2}\int_{0}^{T}\mathrm{Cov}(\langle X(t),u_{\delta,y}\rangle,\langle X(t),w_{\delta,y}\rangle)\diff t-T\psi({u},{w})\right|\to0.
		\]
		\item[$\operatorname{(iii)}$] If, additionally, $\psi({u},{w})=0$, then 
		\[
		\sup_{y\in \mathcal{J}}\left|\delta^{-3}\int_{0}^{T}\mathrm{Cov}(\langle X(t),u_{\delta,y}\rangle,\langle X(t),w_{\delta,y}\rangle)\diff t\right|\lesssim 1.
		\]
	\end{enumerate}
\end{lemma}

\begin{lemma}
	\label{ConvouterVar}
	Grant Assumption \ref{ass: total}, and let $X_0=0$.
	\begin{itemize}
		\item[$\operatorname{(i)}$] 
		For $u,w\in\{D^\alpha K:|\alpha|\leq2\}$,
		\[ 
		\sup_{y\in \mathcal{J}}\mathrm{Var}\left(\int_{0}^{T}\langle X(t),u_{\delta,y}\rangle\langle X(t),w_{\delta,y}\rangle \diff t\right)=O(\delta^6).
		\]
		\item[$\operatorname{(ii)}$]  
		For $u\in\{\partial_iK:1\leq i\leq d\}$ and $w \coloneqq g^{(\theta,y,\delta)}\cdot \nabla K$, with $g^{(\theta,y,\delta)}$ defined in \eqref{eq: functiongtotal} below, it holds
		\[ 
		{\sup_{x\in\mathcal{J}}}\sup_{y\in \mathcal{J},|y-x|\leq h}\mathrm{Var}\left(\int_{0}^{T}\langle X(t),u_{\delta,y}\rangle\langle X(t),w_{\delta,y}\rangle \diff t\right)=O(\delta^4 h^{2\beta}).
		\]
		\item[$\operatorname{(iii)}$] 
		For $u\in\{\partial_iK:1\leq i\leq d\}$ and $w \coloneqq \varphi_\theta(y+\delta\cdot)K$, we have
		\[
		{\sup_{x\in\mathcal{J}}}\sup_{y\in \mathcal{J},|y-x|\leq h}\mathrm{Var}\left(\int_{0}^{T}\langle X(t),u_{\delta,y}\rangle\langle X(t),w_{\delta,y}\rangle \diff t\right)=O(\delta^2 h^{2\beta}).
		\]
	\end{itemize}
\end{lemma}

\subsection{Proof of the upper bound}
\label{sec: detailederror}
Before proving Theorem \ref{thm: maintheorem}, we carefully inspect the observed Fisher information $\mathcal{I}_\delta^x$ and the remainder $\mathcal{R}_\delta^x$ appearing in the decomposition \eqref{eq: est2}.

\subsubsection{The Fisher information and the martingale part}
\label{sec: fisherinfo}
\begin{proposition}
	\label{prop: fisher}
	Grant Assumption \ref{ass: total}. Then, 
	\[
	\mathcal{I}^x_\delta\stackrel{\P}{\rightarrow}\Sigma, \quad \text{ where }\quad
	\Sigma_{ij}=\frac{T}{2a}\sc{(-\Delta)^{-1}\partial_iK}{\partial_jK}, \ i,j\in\{1,\ldots,d\},
	\] and $\Sigma$ thus defined is invertible.
	\begin{proof}
		We only consider the case where $X_0=0$. 
		Note that Assumption \ref{ass: total}(iv) implies the assumed structure in \cite{altmeyer_anisotrop2021}, cf.~\cite[Lemma 2.2]{altmeyer_anisotrop2021}. 
		We hence refer to \cite[Theorem 2.3]{altmeyer_anisotrop2021} for the invertibility of $\Sigma$ and the generalisation of the initial condition. 
		Thus, it suffices to show that, for $1\leq i,j\leq d$,
		\begin{equation}
			\label{eq: propfishersuff}
			\E[\mathcal{I}^x_\delta]_{ij}\rightarrow\Sigma_{ij},\quad \mathrm{Var}((\mathcal{I}^x_\delta)_{ij})\rightarrow0.
		\end{equation}
		Recall that, for $z,z'\in L^2(\R^d)$, the function $\psi(\cdot,\cdot)$ introduced in Lemma \ref{lem: ConvFisher} is defined as
		\[
		\psi(z,z')=\int_{0}^{\infty}\sc{\e^{ta\Delta} z}{\e^{ta\Delta}z'}\diff t=\frac{1}{2a}\sc{(-\Delta)^{-1}z}{z'}.
		\]
		In view of $\sum_{k=1}^Nw_k(x)=1$, $\sum_{k=1}^N|w_k(x)|\lesssim1$, the first part of \eqref{eq: propfishersuff} follows by
		\begin{align*}
			&{\sup_{x\in\mathcal{J}}}|\E[\mathcal{I}_{\delta}^x]_{ij}-\Sigma_{ij}|\\
			&\leq {\sup_{x\in\mathcal{J}}}\sum_{k=1}^{N}|w_k(x)|\left|\delta^{-2}\int_{0}^{T}\mathrm{Cov}(\sc{X(t)}{(\partial_iK)_{\delta,x_k}},\sc{X(t)}{(\partial_jK)_{\delta,x_k}})\diff t-T\psi(\partial_iK,\partial_jK)\right|\\
			&{\leq C_*} \sup_{y\in \mathcal{J}}\left|\delta^{-2}\int_{0}^{T}\mathrm{Cov}(\sc{X(t)}{(\partial_iK)_{\delta,y}},\sc{X(t)}{(\partial_jK)_{\delta,y}})\diff t-T\psi(\partial_iK,\partial_jK)\right|\\
			&\rightarrow0,
		\end{align*}
		where the convergence statement in the last line is a consequence of Lemma \ref{lem: ConvFisher}(ii). 
		By the Cauchy--Schwarz inequality and Lemma \ref{ConvouterVar}(i), we obtain
		\begin{align*}
			{\sup_{x\in\mathcal{J}}}\operatorname{Var}((\mathcal{I}_\delta^x)_{ij})^{1/2}&\leq {\sup_{x\in\mathcal{J}}}\sum_{k=1}^{N}|w_k(x)|\delta^{-2}\mathrm{Var}\left(\int_{0}^{T}\sc{X(t)}{(\partial_iK)_{\delta,x_k}}\sc{X(t)}{(\partial_jK)_{\delta,x_k}}\diff t\right)^{1/2}\\
			&\leq {C_*}\sup_{y\in \mathcal{J}}\delta^{-2}\mathrm{Var}\left(\int_{0}^{T}\sc{X(t)}{(\partial_iK)_{\delta,y}}\sc{X(t)}{(\partial_jK)_{\delta,y}}\diff t\right)^{1/2}\rightarrow0,
		\end{align*}
		concluding the proof.
	\end{proof}
\end{proposition}

\begin{proposition}
	\label{prop: quadvar}
	Grant Assumption \ref{ass: total}. Then,
	\begin{equation*}
		\label{eq: quadrvarord}
		[\mathcal{M}^x_\delta]_T\coloneqq \sum_{k=1}^{N}w_k(x)^2\int_{0}^{T}X^\nabla_{\delta,k}(t)X^\nabla_{\delta,k}(t)^\top \diff t=O_{\P}((Nh^d)^{-1}),
	\end{equation*}
 where the stochastic order of the right hand side is indepenent of $x\in\mathcal{J}.$
	\begin{proof}
		Again, it suffices to verify the claim with initial condition $X_0=0$. We show $\E[[\mathcal{M}^x_\delta]_T]=O((Nh^d)^{-1})$ and $\mathrm{Var}(([\mathcal{M}^x_\delta]_T)_{ij})=o((Nh^d)^{-2})$. Using that $\sum_{k=1}^Nw_k(x)^2\lesssim (Nh^d)^{-1}$, we get similarly to the proof of Proposition \ref{prop: fisher} that
		\begin{align*}
            {\sup_{x\in\mathcal{J}}}\E[[\mathcal{M}_\delta^x]_T]_{ij}&\leq{\sup_{x\in\mathcal{J}}}\sum_{k=1}^{N}w_k(x)^2\left|\delta^{-2}\int_{0}^{T}\mathrm{Cov}(\sc{X(t)}{(\partial_iK)_{\delta,x_k}},\sc{X(t)}{(\partial_jK)_{\delta,x_k}})\diff t\right|\lesssim (Nh^d)^{-1}
		\end{align*}
		as well as
		\begin{align*}
			{\sup_{x\in\mathcal{J}}}\mathrm{Var}(([\mathcal{M}^x_\delta]_T)_{ij})^{1/2}
			&\leq{\sup_{x\in\mathcal{J}}}\sum_{k=1}^{N}w_k(x)^2\sup_{y\in\mathcal{J}}\delta^{-2}\mathrm{Var}\left(\int_{0}^{T}\sc{X(t)}{(\partial_iK)_{\delta,y}}\sc{X(t)}{(\partial_jK)_{\delta,y}}\diff t\right)^{1/2}\\
            &\leq{C_*(Nh^d)^{-1}\sup_{y\in\mathcal{J}}\delta^{-2}\mathrm{Var}\left(\int_{0}^{T}\sc{X(t)}{(\partial_iK)_{\delta,y}}\sc{X(t)}{(\partial_jK)_{\delta,y}}\diff t\right)^{1/2}}\\
			&=o((Nh^d)^{-1}).
		\end{align*}
	\end{proof}
\end{proposition}

\subsubsection{The remainder term}
\label{sec: restterm}
In this subsection, we will study the expected value and variance of the remainder term $\mathcal{R}_\delta^x$, given by 
\[
\mathcal{R}^x_\delta=\sum_{k=1}^{N}w_k(x)\int_{0}^{T}X^\nabla_{\delta,k}(t)\sc{X(t)}{((\theta-\theta(x)) \cdot\nabla+\varphi_\theta)K_{\delta,x_k}}\diff t.
\]
We start by exploring the connection between the weight functions $w_k(x)$ and the multivariate Taylor expansion. 
Define the difference
\begin{equation}
	\label{eq: functiongtotal}
	g^{(\theta,x_k,\delta)}(y)\coloneqq \theta(x_k+\delta y)-\theta(x).
\end{equation}
For $1\leq i\leq d$, its $i$-th entry is given by the first order multivariate Taylor expansion with Peano remainder $P_{1,i,x_k}(y)$
\begin{equation}
	\label{eq: functiong}
	\begin{split}
		g^{(\theta,x_k,\delta)}_i(y)&\coloneqq\sum_{|\alpha|=1}\frac{D^\alpha \theta_i(x)}{\alpha!}(x_k+\delta y-x)^\alpha+\sum_{|\alpha|=1}\frac{D^\alpha (\theta_i(\xi)-\theta_i(x))}{\alpha!}(x_k+\delta y-x)^\alpha\\
		&=\sum_{|\alpha|=1}\frac{D^\alpha \theta_i(x)}{\alpha!}(x_k+\delta y-x)^\alpha+P_{1,i,x_k}(y),
	\end{split}
\end{equation}
for some value $\xi\in B_{x_k+\delta y}(x)$. 
\begin{corollary}
	\label{cor: connweights}
	Grant Assumption \ref{ass: total}. Then, for any $0\leq s\leq T\delta^{-2}$,
	\begin{equation}
		\label{eq: connectionw}
		\sum_{k=1}^{N}w_k(x)\sc{\e^{2sa\Delta}\nabla K}{g^{(\theta,x_k,\delta)}\cdot \nabla K}_{L^2(\R^d)}\lesssim h^\beta(1\wedge s^{-3/2-d/4}).
	\end{equation}
	\begin{proof}
		By Assumption \ref{ass: total}(iii),
		\begin{align*}
			\sum_{k=1}^N w_k(x)g_i^{(\theta,x_k,\delta)}(y)=\sum_{k=1}^Nw_k(x)\left(\sum_{|\alpha|=1}\frac{D^\alpha \theta_i(x)}{\alpha!}(\delta y)^\alpha+P_{1,i,x_k}(y)\right).
		\end{align*}
		Thus,
		\begin{align*}
			&\sum_{k=1}^{N}w_k(x)\sc{\e^{2sa\Delta}\nabla K}{g^{(\theta,x_k,\delta)}\cdot \nabla K}_{L^2(\R^d)}\\
			&\hspace*{1em}=\sum_{k=1}^{N}w_k(x)\sc{\e^{2sa\Delta}\nabla K}{\sum_{i=1}^d\left(\sum_{|\alpha|=1}\frac{D^\alpha \theta_i(x)}{\alpha!}(\delta \cdot)^\alpha+P_{1,i,x_k}\right)\partial_iK}_{L^2(\R^d)}.
		\end{align*}
		By the symmetry of the heat kernel, $\e^{2as\Delta}\partial_j{K}(y)=(q_{2as}*\partial_jK)(y)$ is even if $\partial_jK$ is even and odd if $\partial_jK$ is odd, respectively, and Assumption \ref{ass: total} guarantees that one of these cases always holds true. 
		Moreover, the identity in $\R^d$ is an odd function which implies that $y_l\partial_iK(y)$ is odd if $\partial_iK$ is even and even if $\partial_iK$ is odd, respectively. Hence, for all $1\leq j,l,i\leq d$,
		\[
		\sc{\e^{2sa\Delta}\partial_j K(y)}{\delta y_l\partial_iK(y)}_{L^2(\R^d)}=0
		\]
		as an integral over an odd function. Note that $\norm{P_{1,i,x_k}\partial_iK}_{L^2(\R^d)}\lesssim h^\beta$ whenever $|x_k-x|\leq h$ due to $\delta\ll h$, the Hölder assumption on $\theta$ and $\partial_iK$ having compact support. 
		Indeed, 
		\begin{align*}
			\norm{P_{1,i,x_k}\partial_iK}_{L^2(\R^d)}&\lesssim\norm{|x_k+\delta y-x|^\beta\partial_iK(y)}_{L^2(\R^d)}\\
			&\lesssim |x_k-x|^\beta \norm{\partial_iK}_{L^2(\R^d)}\lesssim h^\beta.
		\end{align*}
		\eqref{eq: connectionw} hence follows by the Cauchy--Schwarz inequality, $K=(-\Delta)\bar{K}$ and Lemma \ref{boundS*u}, since 
		\begin{align*}
			\sum_{k=1}^{N}w_k(x)\sum_{i=1}^d\sc{\e^{2sa\Delta}\nabla K}{g^{(\theta,x_k,\delta)}\cdot \nabla K}_{L^2(\R^d)}
			&=\sum_{k:|x_k-x|\leq h}w_k(x)\sum_{i=1}^d\sc{\e^{2sa\Delta}\nabla K}{g^{(\theta,x_k,\delta)}\cdot \nabla K}_{L^2(\R^d)}\\
			&=\sum_{k:|x_k-x|\leq h}w_k(x)\sc{\e^{2sa\Delta}\nabla K}{\sum_{i=1}^dP_{1,i,x_k}\partial_iK}_{L^2(\R^d)}\\
			&\lesssim h^\beta(1\wedge s^{-3/2-d/4}).
		\end{align*}
	\end{proof}
\end{corollary}

\begin{proposition}\label{prop: expectrest}
	Grant Assumption \ref{ass: total}, and assume that $X_0=0$. 
	Then,
	\[
	{\sup_{x\in\mathcal{J}}}\E[\mathcal{R}_\delta^x]=O(h^\beta).
	\]
	\begin{proof}
		Using the covariance structure in Lemma \ref{lem:covFun} and the rescaling Lemma \ref{rescaledsemigroup}, we obtain
		\begin{align}\nonumber
			\E[\mathcal{R}_\delta^x]&=\sum_{k=1}^{N}w_k(x)\int_{0}^{T}\E\left[X^\nabla_{\delta,k}(t)\sc{X(t)}{((\theta-\theta(x)) \cdot\nabla+\varphi_\theta)K_{\delta,x_k}}\right]\diff t\\\nonumber
			&=\sum_{k=1}^{N}w_k(x)\int_{0}^{T}\mathrm{Cov}(X^\nabla_{\delta,k},\sc{X(t)}{((\theta-\theta(x))\cdot\nabla+\varphi_\theta)K_{\delta,x_k}})\diff t\\\label{eq: expectvaldecomp}
			&=\int_{0}^{T}\int_{0}^{t\delta^{-2}}A(s)\diff s\diff t	+\delta\int_{0}^{T}\int_{0}^{t\delta^{-2}}B(s)\diff s\diff t,
		\end{align}
		with
		\begin{align*}
			A(s)&\coloneqq \sum_{k=1}^{N}w_k(x)\sc{S^\ast_{\theta,\delta,x_k}(s)\nabla K}{S^\ast_{\theta,\delta,x_k}(s)(\theta(x_k+\delta\cdot)-\theta(x))\cdot \nabla K}_{L^2(\Lambda_{\delta,x_k})},\\
			B(s)&\coloneqq \sum_{k=1}^{N}w_k(x)\sc{S^\ast_{\theta,\delta,x_k}(s)\nabla K}{S^\ast_{\theta,\delta,x_k}(s)\varphi_\theta(x_k+\delta\cdot)K}_{L^2(\Lambda_{\delta,x_k})}.		
		\end{align*}
		Noting that ${S}^\ast_{\theta,\delta,x_k}(s)u(x)=0$ for $x\notin\Lambda_{\delta,x_k}$ and using multivariate Taylor expansion for $\theta_i$, we can write  
		\begin{equation*}
			\label{eq: taylorrepres}
			A(s)=\sum_{k=1}^{N}w_k(x)\sc{{S}^*_{\theta,\delta,x_k}(s)\nabla K}{{S}^*_{\theta,\delta,x_k}(s)g^{(\theta,x_k,\delta)}\cdot \nabla K}_{L^2(\R^d)},
		\end{equation*}
		with $g^{(\theta,x_k,\delta)}$ given by \eqref{eq: functiongtotal}. 
		Corollary \ref{cor: connweights} already implies
		\[
		\sum_{k=1}^{N}w_k(x)\sc{\e^{2sa\Delta}\nabla K}{g^{(\theta,x_k,\delta)}\cdot \nabla K}_{L^2(\R^d)}\lesssim h^\beta(1\wedge s^{-3/2-d/4}).
		\]
		Hence,
		\begin{equation}
			\label{eq: error e}
			\int_0^T\int_0^{t\delta^{-2}}\sum_{k=1}^{N}w_k(x)\sc{\e^{2sa\Delta}\nabla K}{g^{(\theta,x_k,\delta)}\cdot \nabla K}_{L^2(\R^d)}\diff s\diff t\lesssim h^\beta.
		\end{equation}
		Thus, it remains to control the error terms resulting from the switch of semigroups. This is given in the next lemma. The proof relies on the $L^2$-distance of $\e^{sa\Delta}$ to $\bar{S}_{\delta,y}(s)$ (pointed out in Lemma \ref{FeynmanKac}(iii)), the $L^2$-distance of $\bar{S}_{\delta,y}(s)$ to $S^*_{\theta,\delta,y}(s)$ (which can be controlled via the variation of parameters formula) and a sufficiently sharp upper bound for $\norm{S^*_{\theta,\delta,y}(s)g^{(\theta,\delta,y)}\cdot\nabla K}_{L^2(\Lambda_{\delta,y})}$. 
		\begin{lemma}
			\label{lem: shiftsemigroup} It holds
			\begin{align*}
				&\int_0^T\int_0^{t\delta^{-2}}A(s)\diff s\diff t=\int_0^T\int_0^{t\delta^{-2}}\sum_{k=1}^{N}w_k(x)\sc{\e^{2sa\Delta}\nabla K}{g^{(\theta,x_k,\delta)}\cdot \nabla K}_{L^2(\R^d)}\diff s\diff t + o(h^\beta),
			\end{align*}
        {where the $o$-term is independent of $x\in\mathcal{J}.$}
		\end{lemma}
		Lemma \ref{lem: shiftsemigroup} combined with \eqref{eq: error e} already yields the desired rate $h^\beta$ for the leading order term $\int_0^T\int_0^{t\delta^{-2}}A(s)\diff s\diff t$ in \eqref{eq: expectvaldecomp}. 
		The lower order term $\delta\int_{0}^{T}\int_{0}^{t\delta^{-2}}B(s)\diff s\diff t$ is bounded	in the same manner. Expand the right-hand side of the scalar product by adding and subtracting $\varphi_\theta(x)$. 
		Following the same structure as above, i.e., switching from the semigroup on $L^2(\Lambda_{\delta,x_k})$ to the heat kernel on $L^2(\R^d)$, we similarly obtain
		\begin{align*}
			&\delta\int_{0}^{T}\int_{0}^{t\delta^{-2}}\sum_{k=1}^{N}w_k(x)\langle{S}^*_{\theta,\delta,x_k}(s)\nabla K,{S}^*_{\theta,\delta,x_k}(s)(\varphi_\theta(x_k+\delta\cdot)-\varphi_\theta(x))K\rangle_{L^2(\Lambda_{\delta,x_k})}\diff s\diff t=o(h^\beta).
		\end{align*}
		On the other hand, using  $\psi(\nabla {K},\varphi_\theta(x){K})=0$ (due to integration by parts), Lemma \ref{lem: ConvFisher}(iii), Lemma \ref{ConvouterVar}(i) and \eqref{eq: order h}, we derive
		\begin{align*}
			&\delta\int_{0}^{T}\int_{0}^{t\delta^{-2}}\sum_{k=1}^{N}w_k(x)\sc{{S}^*_{\theta,\delta,x_k}(s)\nabla {K}}{{S}^*_{\theta,\delta,x_k}(s)\varphi_\theta(x)K}_{L^2(\Lambda_{\delta,x_k})}\diff s\diff t\lesssim\delta^2=o(h^\beta).
		\end{align*}
	\end{proof}
\end{proposition}

\begin{proposition}
	\label{prop: varrest}
	Grant Assumption \ref{ass: total}, and suppose that $X_0=0$. Then,
	\[ {\sup_{x\in\mathcal{J}}}\mathrm{Var}(\mathcal{R}_\delta^x)=O(h^{2\beta}).\]
	\begin{proof}
		We will show that each entry of the covariance matrix of $\mathcal{R}_\delta^x$ satisfies the required order, which then directly implies the order for the entire covariance matrix.
		Note that $	\mathrm{Cov}(\mathcal{R}_{\delta}^x)_{ij}$ is given by
		\begin{align*}
			&\sum_{k=1}^{N}\sum_{l=1}^{N}w_k(x)w_l(x)\delta^{-2} \mathrm{Cov}\left(\int_{0}^{T}\sc{X(t)}{(\partial_iK)_{\delta,x_k}}\sc{X(t)}{(\varphi_\theta+(\theta-\theta(x)) \cdot\nabla)K_{\delta,x_k}}\diff t,\right.\\
			&\hspace*{12em}	\left.\int_{0}^{T}\sc{X(t)}{(\partial_jK)_{\delta,x_l}}\sc{X(t)}{(\varphi_\theta+(\theta-\theta(x)) \cdot\nabla)K_{\delta,x_l}}\diff t\right).
		\end{align*}
		The Cauchy--Schwarz inequality and $(a+b)^2\leq 2a^2+2b^2$ imply that, up to constants {independent of $x\in\mathcal{J}$}, this last quantity is upper bounded by
		\begin{align*}
			&\delta^{-2}\sup_{y\in \mathcal{J},|y-x|\leq h,k\leq d}\mathrm{Var}\left(\int_{0}^{T}\sc{X(t)}{(\partial_kK)_{\delta,y}}\sc{X(t)}{(\varphi_\theta(y+\delta\cdot)K)_{\delta,y}}\diff t\right)\\
			&\quad+ \delta^{-4}\sup_{y\in \mathcal{J},|y-x|\leq h,k\leq d}\mathrm{Var}\left(\int_{0}^{T}\sc{X(t)}{(\partial_kK)_{\delta,y}}\sc{X(t)}{g^{(\theta,y,\delta)} \cdot\nabla K)_{\delta,y}}\diff t\right),
		\end{align*}
		with $g^{(\theta,y,\delta)}$ from \eqref{eq: functiongtotal}. 
		The result follows then immediately by Lemma \ref{ConvouterVar}(ii) and (iii).
	\end{proof}
\end{proposition}

Proposition \ref{prop: expectrest} and \ref{prop: varrest} already imply that $\mathcal{R}_\delta^x$ is of stochastic order $O_{\P}(h^\beta)$ whenever $X_0=0$. 
Under Assumption \ref{ass: total}(iv), this can furthermore be extended to general initial conditions. 

\begin{proposition}\label{prop: splittingrest}
	Grant Assumption \ref{ass: total}. 
	Define $\bar{\mathcal{R}}_\delta^x$ analogous to $\mathcal{R}_\delta^x$, but with respect to $\bar{X}$ satisfying \eqref{eq: SPDE} with initial condition $\bar{X}(0)=0$. 
	Then,
	\begin{equation*}
		\mathcal{R}^x_\delta=\bar{\mathcal{R}}_\delta^x+o_{\P}(h^\beta),
	\end{equation*}
 {where the $o_{\P}$-term does not depend on $x\in\mathcal{J}.$}
\end{proposition}

\subsubsection{Proof of the upper bound statement}

\begin{proof}[Proof of Theorem \ref{thm: maintheorem}]
	We use the error decomposition \eqref{eq: est2}. 
	{To prove \eqref{eq: mainresult}, it} suffices to show for $\delta\to0$ that $\mathcal{I}^x_\delta\stackrel{\P}{\to}\Sigma$ for some invertible, deterministic matrix $\Sigma$, while $\mathcal{M}^x_\delta=O_{\P}((Nh^d)^{-1/2})$ and $\mathcal{R}^x_\delta=O_{\P}(h^\beta)$.
	Proposition \ref{prop: fisher} gives that $\mathcal{I}^x_\delta\stackrel{\P}{\to}\Sigma$ for some invertible $\Sigma$. 
	Define a sequence of martingales via
	\[
	\mathcal{M}^x_\delta(t)=\sum_{k=1}^{N}w_k(x)\int_{0}^{t}X^\nabla_{\delta,k}(s)\diff W_k(s).
	\]
	In particular, due to the independence of the Brownian motions $W_k$ guaranteed by Assumption \ref{ass: total}, the quadratic variation of $\mathcal{M}^x_\delta=\mathcal{M}^x_\delta(T)$ is given by
	\[
	[\mathcal{M}^x_\delta]_T=\sum_{k=1}^{N}w_k(x)^2\int_{0}^{T}X^\nabla_{\delta,k}(t)X^\nabla_{\delta,k}(t)^\top \diff t.
	\]
	A standard argument, cf.~\cite[Lemma 3.6 or Lemma 3.8]{whitt_martingale_FCLT_2007}, shows that $\mathcal{M}^x_\delta$ behaves like the squared root of its quadratic variation, i.e., using Proposition \ref{prop: quadvar}, $\mathcal{M}^x_\delta=O_{\P}((Nh^d)^{-1/2})$.
	Combining Proposition \ref{prop: expectrest}, Proposition \ref{prop: varrest} and Proposition \ref{prop: splittingrest} yields the rate $O_{\P}(h^\beta)$ for $\mathcal{R}_\delta^x$.

    To prove the supplement \eqref{eq: extraresult}, it is enough to show that 
    \begin{align}
        \P\left(\big|(\mathcal{I}_\delta^x)^{-1}\mathcal{R}_\delta^x\big|h^{-\beta}>M\right)\leq\frac{\varepsilon}{2},\label{eq: extra1}\\
        \P\left(\big|(\mathcal{I}_\delta^x)^{-1}\mathcal{M}_\delta^x\big|(Nh^d)^{1/2} >M\norm{K}^{-1}_{L^2(\R^d)}\right)\leq\frac{\varepsilon}{2}.\label{eq: extra2}
    \end{align}
    We only show the statement \eqref{eq: extra1}, as the arguments for \eqref{eq: extra2} are similar. 
    Now,
    \begin{align}
        \P\left(\big|(\mathcal{I}_\delta^x)^{-1}\mathcal{R}_\delta^x\big|h^{-\beta}>M\right)&\leq 
        \P\left(\big|((\mathcal{I}_\delta^x)^{-1}-\Sigma^{-1})\mathcal{R}_\delta^x\big|h^{-\beta}>M\right)
        +\P\left(\big|\Sigma^{-1}\mathcal{R}_\delta^x\big|h^{-\beta}>M\right)\nonumber\\
        &\leq \P\left(\norm{(\mathcal{I}_\delta^x)^{-1}-\Sigma^{-1}}\big|\mathcal{R}_\delta^x\big|h^{-\beta}>M\right)+P\left(\big|\mathcal{R}_\delta^x\big|h^{-\beta}>M\norm{\Sigma^{-1}}^{-1}\right)\label{eq: extra3}
    \end{align}
    with arbitrary matrix norm $\norm{\cdot}$ on $\R^{d\times d}$. Due to Proposition \ref{prop: expectrest}, Proposition \ref{prop: varrest}, Chebyshev's inequality, and for $\delta$ sufficiently small and $M$ sufficiently large, the term $\P\left(\big|\mathcal{R}_\delta^x\big|h^{-\beta}>M\norm{\Sigma^{-1}}^{-1}\right)$ is uniformly bounded in $x\in\mathcal{J}$ by $\varepsilon/4$. 
    On the other hand, 
    \begin{align*}
        P\left(\norm{(\mathcal{I}_\delta^x)^{-1}-\Sigma^{-1}}\left|\mathcal{R}_\delta^x\right|h^{-\beta}>M\right)\leq P\left(\left|\mathcal{R}_\delta^x\right|h^{-\beta}>M\right)+P\left(\norm{(\mathcal{I}_\delta^x)^{-1}-\Sigma^{-1}}>1\right).
    \end{align*}
    Again, $\delta$ and $M$ can be chosen such that $P\left(\left|\mathcal{R}_\delta^x\right|h^{-\beta}>M\right)\leq\varepsilon/8$. 
    Moreover, there exists a value $\eta>0$ with the property that $\norm{((\mathcal{I}_\delta^x)^{-1}-\Sigma^{-1})}\leq 1$ whenever $\norm{(\mathcal{I}_\delta^x-\Sigma)}\leq \eta$, due to the continuity of the function $y\mapsto y^{-1}$ and the fact that both $\Sigma$ and $\mathcal{I}_\delta^x$ are (a.s.) invertible. Hence, for sufficiently small $\delta$,
    \begin{align*}
        \P\left(\norm{(\mathcal{I}_\delta^x)^{-1}-\Sigma^{-1}}>1\right)\leq P\left(\norm{\mathcal{I}_\delta^x-\Sigma}> \eta\right)\leq\varepsilon/8
    \end{align*}
    due to Proposition \ref{prop: fisher}, thus showing the assertion.
\end{proof}

\subsection{Proof of the lower bound}
\label{sec:prooflower}
The proof of Theorem \ref{thm:lower:bound:M>1} relies on the general reduction scheme in \cite[Section 2.2]{tsybakov_introduction_2008} and the RKHS machinery described in detail in \cite[Section 6.3]{altmeyer_anisotrop2021}. 
In what follows, we will therefore summarise the key components until the nonparametric setup requires a different reasoning.

Let $\mathbb{P}_{\vartheta^0}$ and $\mathbb{P}_{\vartheta^1}$ be two Gaussian measures defined on a separable Hilbert space $\mathcal{H}$ with expectation zero and positive self-adjoint trace-class covariance operators $C_{\vartheta^0}$ and $C_{\vartheta^1}$, respectively. $\theta^0$ and $\theta^1$ belong to a set of functions $\Theta$.
By the spectral theorem, there exist (strictly) positive eigenvalues $(\sigma_j^2)_{j\geq 1}$ and an associated orthonormal system of eigenvectors $(u_j)_{j\geq 1}$ such that $C_{\vartheta^0}=\sum_{j\geq 1}\sigma_j^2(u_j\otimes u_j)$. 
The reproducing kernel Hilbert space (RKHS) associated to $\mathbb{P}_{\vartheta^0}$ is given by
\begin{align*}
	H_{\theta^0}=\{h\in \mathcal{H}:\|h\|_{H_{\vartheta^0}}<\infty\},\qquad \|h\|_{H_{\vartheta^0}}^2=\sum_{j\geq 1}\frac{\sc{u_j}{h}_{\mathcal{H}}^2}{\sigma_j^2}.
\end{align*}
Instead of \cite[Lemma 6.8]{altmeyer_anisotrop2021}, we rely on its nonparametric equivalent. 
The proof is identical and therefore skipped.
\begin{lemma}\label{lem:Gaussian:lower:bound}
	In the above Gaussian setting, suppose that $(u_j)_{j\geq 1}$ is an orthonormal basis of $\mathcal{H}$ and that 
	\begin{align}\label{eq:Gaussian:lower:bound:condition}
		\sum_{j\geq 1}\sigma_j^{-2}\|(C_{\vartheta^1}-C_{\vartheta^0})u_{j}\|_{H_{\vartheta^0}}^2
		\leq \frac12.
	\end{align}
	Then, the squared Hellinger distance satisfies the bound $H^2(\mathbb{P}_{\theta^0},\mathbb{P}_{\vartheta^1})\leq 1$. Therefore, for any $x\in\Lambda$ and a generic constant $c_1>0$,
	\begin{align*}
		\inf_{\hat\vartheta}\max_{\vartheta\in\{\vartheta^0,\vartheta^1\}}\mathbb{P}_\vartheta\left(|\hat\vartheta(x)-\vartheta(x)|\geq \frac{c_1N^{-\beta/(2\beta+d)}}{2}\right)\geq \frac{1}{4}\cdot\frac{2-\sqrt{3}}{4}\eqqcolon c_2.
	\end{align*}
\end{lemma}
We assume without loss of generality that $\norm{K}_{L^2(\R^d)}=1$.
Choose $\theta^0$ such that the null model is $A_{\theta^0}=\Delta$, i.e., $\theta=0$, $c=0$, and choose $\theta^1$ such that the alternatives are  $A_{\theta^1}=\Delta+\theta\cdot\nabla+c$, where $c-\nabla\cdot\theta\leq\gamma<0$ and $\theta$ is componentwise $\beta$-Hölder continuous and a conservative vector field. 
For $\vartheta\in\{\vartheta^0,\vartheta^1\}$, let $\P_{\theta,\delta}$ be the law of $X_{\delta}$ on $\mathcal{H}=L^2([0,T])^M$, let  $C_{\vartheta,\delta}$ be its covariance operator, and let $(H_{\vartheta,\delta},\norm{\cdot}_{H_{\vartheta,\delta}})$ be the associated RKHS. 
For $(f_k)_{k=1}^M\in\mathcal{H}$, we have $C_{\theta,\delta}(f_k)_{k=1}^M=(\sum_{l=1}^MC_{\theta,\delta,k,l}f_l)_{k=1}^M$ with (cross-) covariance operators $C_{\theta,\delta,k,l}\colon L^2([0,T])\to L^2([0,T])$ defined by
\begin{align*}
	C_{\theta,\delta,k,l}f_l(t)=\E_{\theta} [\sc{X_{\delta,l}}{f_l}_{L^2([0,T])}X_{\delta,k}(t)],\qquad 0\leq t\leq T.
\end{align*}
Due to stationarity of $X_\delta$ (cf.~Assumption \ref{ass:lowerBound}), we have, for $0\leq t\leq T,$
\begin{align*}
	C_{\theta,\delta,k,l}f_l(t)
	&=\int_0^t c_{\theta,\delta,k,l}(t-t')f_l(t')\,\diff t'+\int_t^T c_{\theta,\delta,l,k}(t'-t)f_l(t')\,\diff t',
\end{align*}
with covariance kernels $c_{\theta,\delta,k,l}(t)= \E_{\theta} [X_{\delta,k}(t)X_{\delta,l}(0)]$, $0\leq t\leq T$. \\

\smallskip

Let $(\sigma_j^2)_{j\geq 1}$ be the strictly positive eigenvalues of $C_{\vartheta^0,\delta}$, and let $(u_j)_{j\geq 1}$ with $u_j=(u_{j,k})_{k=1}^M\in \mathcal{H}$ be a corresponding orthonormal system of eigenvectors. 
We want to verify the assumption in \eqref{eq:Gaussian:lower:bound:condition}, for which we require the following lemma.

\begin{lemma}[Lemma 6.9 in \cite{altmeyer_anisotrop2021}]\label{lem:seriesBound:M>1}
	In the above setting, we have
	\begin{align*}
		&\sum_{j=1}^\infty \sigma_j^{-2}\norm{(C_{\theta^0,\delta}-C_{\theta^1,\delta}) u_{j}}^2_{H_{\vartheta^0,\delta}}\\
		&\hspace*{1em}\leq CT  \sum_{k,l=1}^N\Big(\frac{\norm{\Delta K}^4_{L^2(\R^d)}}{\delta^{8}}\norm{c_{\theta^0,\delta,k,l}-c_{\theta^1,\delta,k,l}}^2_{L^2([0,T])}+\norm{c_{\theta^0,\delta,k,l}''-c_{\theta^1,\delta,k,l}''}^2_{L^2([0,T])}\Big)
	\end{align*}
	for all $\delta^2\leq \norm{\Delta K}_{L^2(\R^d)}$ and all $T\geq 1$, where $C>0$ is an absolute constant.
\end{lemma}
Adapting \cite[Lemma 6.10]{altmeyer_anisotrop2021} to our setting results in another upper bound. 
\begin{lemma}\label{lem:concrete_lower_bound}
	In the above setting, let $\theta^1\in\Theta$ with $N\geq1$. 
	Then, there exists a constant $c_3>0$, depending only on $K$ and $d$, such that
	\[
		\sum_{k, l=1}^N\left(\delta^{-8}\norm{c_{\theta^0,\delta,k,l}-c_{\theta^1,\delta,k,l}}^2_{L^2([0,T])}+\norm{c_{\theta^0,\delta,k,l}''-c_{\theta^1,\delta,k,l}''}^2_{L^2([0,T])}\right)\le c_3\sum_{k=1}^N\left(|\theta(x_k)|^2+\delta^2\tilde{c}_\theta(x_k)^2\right),
	\]
	with $\tilde{c}_\theta=c-\nabla\cdot\theta /2-|\theta|^2/4$.
\end{lemma}
Let $c_4,c_5>0$ be constants independent of $N$ and $h$. 
Consider a kernel function $V\in C^\infty_c(\R^d;[0,\infty])$ with compact support in $[-1/2,1/2]^d$. 
Define the potential $\xi(y)=c_4h^{\beta+1} V((y-x)/h)$, and let $\theta=\nabla\xi$.
We consider hence the alternative
$$\theta(y)=c_4h^\beta (\nabla V)\left(\frac{y-x}{h}\right)$$
and a reaction function $c\colon\Lambda\to\R_-$ small enough. 
For $h=c_5N^{-1/(2\beta+d)}$, we have that 
$$c_3\sum_{k=1}^N\left(|\theta(x_k)|^2+\delta^2\tilde{c}_\theta(x_k)^2\right)\lesssim \sum_{k=1}^N|\theta(x_k)|^2\lesssim Nh^{d}h^{2\beta}\lesssim1.$$
The claim of Theorem \ref{thm:lower:bound:M>1} follows now from Lemma \ref{lem:Gaussian:lower:bound} in combination with Lemmas \ref{lem:seriesBound:M>1} and \ref{lem:concrete_lower_bound} and sufficiently small constants $c_4,c_5$.
\qed

\subsection{Remaining proofs}\label{app:additional proofs}
\subsubsection{Remaining proofs for Section \ref{subs: pointwise}}
\begin{proof}[Proof of Corollary \ref{cor: unknown a}]
	We decompose
	$$\tilde{\theta}_\delta(x)=\theta(x)-(\mathcal{I}_\delta^x)^{-1}\mathcal{M}_\delta^x\norm{K}_{L^2(\R^d)}+(\mathcal{I}_\delta^x)^{-1}\mathcal{R}_\delta^x+(\hat{a}_\delta-a)(\mathcal{I}_\delta^x)^{-1}\mathcal{U}_\delta^x$$
	with
	$$\mathcal{U}_\delta^x=\sum_{k=1}^Nw_k(x)\int_0^TX^\nabla_{\delta,k}(t)X^\Delta_{\delta,k}(t)\diff t.$$
	Combining Lemma \ref{lem: ConvFisher}(iii) and Lemma \ref{ConvouterVar}(i), it follows by the arguments given in Section \ref{sec: detailederror} that $\mathcal{U}_\delta^x=O_{\P}(1)$. 
	Thus, the claim hold once $\hat{a}_\delta$ satisfies \eqref{eq: orda}. 
	Just as the estimator $\hat{\theta}_\delta(x)$ described in \eqref{eq: est2}, the estimates in \eqref{eq: esta} can again be decomposed into a bias and martingale part. 
	While the orders of the appearing coefficients differ due to a different scaling in $\delta$, all terms can be controlled with the techniques used in Section \ref{sec: fisherinfo} and \ref{sec: restterm}. 
	It is therefore straightforward to verify that both given candidates for $\hat{a}_\delta$ satisfy
	$$\hat{a}_\delta-a\in O_{\P}(\delta h+\delta^2+\delta(Nh^d)^{-1/2})$$
	and thus fulfill \eqref{eq: orda}.
\end{proof}
\begin{proof}[Proof of Corollary \ref{cor: intrisk}]
	By decomposing the integral and {using \eqref{eq: extraresult} }from Theorem \ref{thm: maintheorem}, we obtain 
	\begin{align*}
		\int_\Lambda\left(\hat{\theta}_\delta(x)-\theta(x)\right)^2\diff x&=\int_{\mathcal{J}}\left(\hat{\theta}_\delta(x)-\theta(x)\right)^2\diff x+\int_{\Lambda\setminus\mathcal{J}}\left(\hat{\theta}_\delta(x)-\theta(x)\right)^2\diff x\\
		&=O_{\P}\left(h^{2\beta}+\frac{1}{Nh^d}\right)+\int_{\Lambda\setminus\mathcal{J}}\left(\hat{\theta}_\delta(x)-\theta(x)\right)^2\diff x.
	\end{align*}
	Due to the decomposition \eqref{eq: est2} and \eqref{eq: estwhole}, it holds for $x\notin\mathcal{J}$ and appropriate $x_0=x_0(x)\in\mathcal{J}$ that
	$$\hat{\theta}_\delta(x)=\hat{\theta}_\delta(x_0)=\theta(x)+(\theta(x_0)-\theta(x))+O_{\P}(h^\beta+(Nh^d)^{-1/2}).$$
	Thus, plugging this into the previous display yields by the Hölder regularity of $\theta$,
	\begin{align*}
		\int_{\Lambda\setminus\mathcal{J}}\left(\hat{\theta}_\delta(x)-\theta(x)\right)^2\diff x&\lesssim O_{\P}\left(h^{2\beta}+\frac{1}{Nh^d}\right)
		+\int_{\Lambda\setminus\mathcal{J}}\left(\theta(x)-\theta(x_0)\right)^2\diff x\\
		&\lesssim O_{\P}\left(h^{2\beta}+\frac{1}{Nh^d}\right)+\int_{\Lambda\setminus\mathcal{J}}\textrm{dist}^2(x,\mathcal{J})\diff x\\
		&\lesssim O_{\P}\left(h^{2\beta}+\frac{1}{Nh^d}\right)+d^2_{\max}\lambda(\Lambda\setminus\mathcal{J}).
	\end{align*}
\end{proof}
\begin{proof}[Proof of Lemma \ref{lem: existweight}]
	We use the well-known theory for local polynomial estimators, more specifically, for the local linear case. 
	The one-dimensional case in \cite[Chapter 1.6]{tsybakov_introduction_2008} can be easily extended to the general $d$-dimensional version. 
	By a first order multivariate Taylor expansion for a function $f\colon \R^d\rightarrow\R$, we can write for $y,z\in\R^d$, a multiindex $\alpha$, and any $h>0$,
	$$f(z)\approx\sum_{0\leq|\alpha|\leq 1}\frac{D^\alpha f(y)}{\alpha!}(z-y)^\alpha=\xi^\top(y)U\left(\frac{z-y}{h}\right),$$
	where
	\begin{equation*}
		\label{eq: taylorU}
		U(u)=\left((u^\alpha/\alpha!)_{0\leq|\alpha|\leq 1}\right)^\top,\quad\xi(x)=
		\left((D^\alpha f(x)h^{|\alpha|})_{0\leq|\alpha|\leq 1}\right)^\top.
	\end{equation*}
	Modifying \cite[Proposition 1.12]{tsybakov_introduction_2008} and \cite[Lemma 1.3]{tsybakov_introduction_2008} to their multivariate counterparts, it follows that the weights $w_k(x)$ are reproducing of order 1 and satisfy Assumption \ref{ass: total}(iii) if (LP1)-(LP3) hold true.
\end{proof}

\subsubsection{Remaining proofs for Section \ref{sec: rescaledsemigroup}}
\begin{proof}[Proof of Lemma \ref{boundS*_01}]
	Since $A_{\theta}$ is elliptic, it follows as in the proof of \cite[Proposition A.4]{altmeyer_nonparametric_2020}, after formally replacing $\Delta_{\theta(\delta\cdot)}$ and $\min_{x}\theta(x)$ contained there by $a\Delta$ and the lower bound on the spectrum of $a$, respectively, that $A^\ast_{\theta,\delta,y}$ is a sectorial operator on $L^2(\Lambda_{\delta,y})$, that is, there exists a constant $M$, independent of $\delta$ and $y\in\mathcal{J}$, such that
	\begin{equation*}
		\norm{(\lambda I-A^\ast_{\theta,\delta,y})^{-1}}_{L^2(\Lambda_{\delta,y})} \leq \frac{M}{|\lambda-C\delta^2|}
	\end{equation*}
	for all $\lambda\in\Sigma_{\eta}=\{\rho\in\mathbb{C}:|\text{arg}(\rho-C\delta^2)|<\eta\}\setminus \{C\delta^2\}$ with some $\eta\in(\pi/2,\pi)$ or, equivalently, for all $\lambda\in\Sigma_{\eta}+C\delta^2$,
	\begin{equation*}
		\norm{(\lambda I+(C\delta^2-A^\ast_{\theta,\delta,y}))^{-1}}_{L^2(\Lambda_{\delta,y})} \leq \frac{M}{|\lambda|}.
	\end{equation*}
	The shifted operator $C\delta^2-A^\ast_{\theta,\delta,y}$ generates the semigroup $\e^{-C\delta^2t}S^*_{\theta,\delta,y}(t)$, and so the result follows from \cite[Proposition 2.1.1]{lunardi_analytic_1995}.
\end{proof}

\begin{proof}[Proof of Lemma \ref{FeynmanKac}]
	The proof is a combination of \cite[Proposition 3.5]{altmeyer_nonparametric_2020} and \cite[Lemma 6.2]{altmeyer_anisotrop2021}. 
	For fixed $y\in \mathcal{J}$, $u\in \R^d$, it holds by a Feynman--Kac representation that
	\begin{equation*}
		\label{representation S}
		S^\ast_{\vartheta,\delta,y}(t)z(u)=\tilde{\mathbb{E}}_u\left[z(Y^{(\delta,y)}_t)\exp\left(\int_{0}^{t}\tilde{c}_{\delta,y}(Y_s^{(\delta,y)})\diff s\right)\mathbf{1}\left(t<\tau_{\delta,y}(Y^{(\delta,y)})\right)\right],
	\end{equation*}
	where the process $Y^{(\delta,y)}$ takes  the form 
	\[
	\diff Y^{(\delta,y)}_t=\tilde{b}_{\delta,y}(Y_t^{(\delta,y)}) \diff t+\sqrt{2}a^{1/2}\diff\tilde{W}_t,\hspace*{2mm}Y^{(\delta,y)}_0=u\in\mathbb{R}^d,
	\]
	with $\tilde{b}_{\delta,y}(\cdot)=-\delta \theta(y+\delta\cdot)$, $\tilde{c}_{\delta,y}(\cdot)=\delta^2(c(y+\delta\cdot)-\nabla\cdot\theta(y+\delta\cdot))$, a scalar Brownian motion $\tilde{W}$, and with the stopping times
	$\tau_{\delta,y}\coloneqq\inf\{t\geq0:Y_t^{(\delta,y)}\notin\Lambda_{\delta,y}\}$.  
	
	(i). By  upper bounding the transition densities of $Y^{(\delta,y)}$ as in \cite[Proposition 3.5(i)]{altmeyer_nonparametric_2020}, we get
	$$\sup_{y\in\mathcal{J}}(S^\ast_{\theta,\delta,y}(t)|z|)(u)\leq c_3\e^{c_1t\delta^{2}}(\e^{c_2t\Delta}|z|)(u),$$
	where the right hand side is in $L^2(\R^d)$.
	
	(ii).  By dense approximation, it is enough to consider $z\in C_c(\bar{\Lambda})$ and such that $z$ is supported in $\Lambda_{\delta,y}$ for $\delta$ small enough, hence, $z|_{\Lambda_{\delta,y}}=z$. 
	With $(\e^{ta\Delta}z)(u)=\tilde{\mathbb{E}}_u[z(Y_t^{(0)})]$, decompose
	\[
	S^\ast_{\vartheta,\delta,y}(t)z(u)-\e^{ta\Delta}z(u)= T_1(y,u)+T_2(y,u)+T_3(y,u)
	\]
	with
	\begin{align*}
		T_1(y,u)&\coloneqq \tilde{\mathbb{E}}_u\left[z(Y^{(\delta,y)}_t)-z(Y^{(0)}_t)\right],\\
		T_2(y,u)&\coloneqq \tilde{\mathbb{E}}_u\left[z(Y^{(\delta,y)}_t)\left(\exp\left(\int_{0}^{t}\tilde{c}_{\delta,y}(Y_s^{(\delta,y)})\diff s\right)-1\right)\mathbf{1}(t<\tau_{\delta,y}(Y^{(\delta,y)}))\right],\\
		T_3(y,u)&\coloneqq -\tilde{\mathbb{E}}_u\left[z(Y^{(\delta,y)}_t)\mathbf{1}(t\geq \tau_{\delta,y}(Y^{(\delta,y)}))\right].
	\end{align*} 
	The arguments in \cite[Proposition 3.5(ii)]{altmeyer_nonparametric_2020} yield \begin{equation*}\sup_{y\in\mathcal{J}}|T_1(y,u)|\rightarrow0\quad \text{ and } \quad\sup_{y\in\mathcal{J}}|T_2(y,u)|\rightarrow0,
	\end{equation*}
	while compactness of $\mathcal{J}$ guarantees for sufficiently small $\delta$ the existence of a ball $B_{\rho\delta^{-1}}\subset\bigcap_{y\in\mathcal{J}}\Lambda_{\delta,y}$ with centre $0$ and radius $\rho\delta^{-1}$ for some $\rho>0$. 
	Using that the running maximum of a Brownian motion decays exponentially, see, for instance \cite[Problem 2.8.3]{karatzas_brownian_1998}, we conclude similarly to \cite[Lemma 6.2(ii)]{altmeyer_anisotrop2021} that
	\begin{align*}
		\label{eq: pointwFeynm}
		\begin{split}
			\sup_{y\in\mathcal{J}}|T_3(y,u)|&=\sup_{y\in \mathcal{J}}|\tilde{\mathbb{E}}_u\left[z(Y_t)\mathbf{1}(t\geq \tau_{\delta,y}(Y))\right]|\\
			&\lesssim \sup_{y\in\mathcal{J}}\tilde{\mathbb{P}}_{u}(\tau_{\delta,y}(Y)\leq t)\leq \tilde{\mathbb{P}}_{u}(\max_{0\leq s\leq t} |Y_s| \geq \rho\delta^{-1})\\ 
			& \leq \tilde{\mathbb{P}}_{u}(\max_{0\leq s\leq t} |\tilde{W}_s| \geq \tilde{\rho}\delta^{-1})\leq \delta t^{1/2} C\e^{-C\delta^{-2}t^{-1}}\to 0,
		\end{split}
	\end{align*}
	for a modified constant $\tilde{\rho}$. 
	This implies pointwise, for all $u\in \R^d$, 
	\begin{equation*}
		\sup_{y\in \mathcal{J}}|S^\ast_{\vartheta,\delta,y}(t)z(u)- \e^{ta\Delta}z(u)|\rightarrow 0,\quad \delta\rightarrow 0.\label{eq:semGroupConvergence}
	\end{equation*}
	By (i), we know $\sup_{y\in\mathcal{J}}|(S^*_{\theta,\delta,y}(t)z)(u)|\in L^2(\R^d)$. Dominated convergence yields the claim.
	
	(iii).  We use the decomposition in (ii). The process $Y^{(\delta,y)}$ is independent of $\delta$ and $\tilde{b}_{\delta,y}=0$, $\tilde{c}_{\delta,y}=0$. This implies $T_1(y,u)=T_2(y,u)=0$ for all $y\in\mathcal{J}$ and $u\in\R^d$. Hölder's inequality thus yields 
	\begin{align*}
		\sup_{y\in\mathcal{J}}\norm{(\bar{S}_{\delta,y}(t)-\e^{ta\Delta})z}_{L^2(\R^d)}
		&\leq\sup_{y\in\mathcal{J}}\left(\norm{(\bar{S}_{\delta,y}(t)-\e^{ta\Delta})z}_{L^1(\R^d)}\norm{(\bar{S}_{\delta,y}(t)-\e^{ta\Delta})z}_{L^\infty(\R^d)}\right)^{1/2}\\
		&\lesssim \delta^{1/2}t^{1/4}\e^{-\delta^{-2}t^{-1}/2}.
	\end{align*}
\end{proof}

\begin{proof}[Proof of Lemma \ref{boundS*u}]
	While the result matches \cite[Lemma 6.4]{altmeyer_anisotrop2021}, the proof differs as we cannot rely on diagonalisability of $S^*_{\theta,\delta,y}(t)$ in the nonparametric framework.
	
	We write $u=V_{\delta,y}z$. 
	Let first $s=0$ such that $H_0^{-s,p}(\R^d)=L^p(\R^d)$. 
	Approximating $u$ by continuous and compactly supported functions, we obtain by Lemma \ref{FeynmanKac}(i) and hypercontractivity of the heat kernel on $\mathbb{R}^d$ uniformly in $y\in\mathcal{J}$
	\begin{align*}
		\norm{S^*_{\vartheta,\delta,y}(t) u }_{L^2(\Lambda_{\delta,y})}
		& \lesssim \e^{c_1t\delta^2} \norm{ \e^{Ct\Delta}|u| }_{L^2(\R^d)}\\
		& \lesssim  \e^{c_1t\delta^2}t^{-\gamma}\norm{u}_{L^p(\mathbb{R}^d)}\lesssim \e^{c_1t\delta^2}t^{-\gamma}\norm{z}_{L^p(\R^d)}.
	\end{align*}
	This yields the result for $s=0$. These inequalities hold also for $p=1$, thus proving the supplement of the statement. 
	For $s>0$ and $p>0$, we apply first Lemma \ref{boundS*_s} and then the inequality from the last display to $(2C\delta^2-A^*_{\theta,\delta,y})^{-s/2}u$ instead of $u$.
	Thus, uniformly in $y\in\mathcal{J}$,
	\begin{align*}
		\norm{S^*_{\vartheta,\delta,y}(t) u}_{L^2(\Lambda_{\delta,y})}
		& = \norm{(2C\delta^2-A^*_{\theta,\delta,y})^{s/2} S^*_{\vartheta,\delta,y}(t) (2C\delta^2-A^*_{\theta,\delta,y})^{-s/2} u}_{L^2(\Lambda_{\delta,y})}\\
		& \lesssim  \e^{c_1t\delta^{-2}}t^{-s/2}\norm{S^*_{\vartheta,\delta,y}(t) (2C\delta^2-A^*_{\vartheta,\delta,y})^{-s/2} u}_{L^2(\Lambda_{\delta,y})} \\
		& \lesssim \e^{c_1t\delta^{-2}}t^{-s/2-\gamma} \norm{(2C\delta^2-A^*_{\vartheta,\delta,y})^{-s/2} u}\norm_{L^{p}(\Lambda_{\delta,y})}\\
		& \lesssim \e^{c_1t\delta^{-2}}t^{-s/2-\gamma} \norm{(-\bar{A}_{\vartheta,\delta,y})^{-s/2} u}_{L^{p}(\Lambda_{\delta,y})}\\
		&\lesssim \e^{c_1t\delta^{-2}}t^{-s/2-\gamma}\norm{u}_{H^{-s,p}(\Lambda_{\delta,y})}\\
		&\lesssim \e^{c_1t\delta^{-2}}t^{-s/2-\gamma}V_{\operatorname{op}}\norm{z}_{L^p(\Lambda_{\delta,y})}.
	\end{align*}
\end{proof}

\subsubsection{Remaining proofs for Section \ref{sec: properties measurements}}
\begin{proof}[Proof of Lemma \ref{lem: ConvFisher}]
	Lemma \ref{boundS*u} applied for $s=2$ shows that, for $v\in\{u,w\}$ and any $\varepsilon>0$, 
	\begin{equation}
		\sup_{y\in\mathcal{J}}\norm{S^\ast_{\vartheta,\delta,y}(t) v}_{L^2(\Lambda_{\delta,y})} \lesssim_\varepsilon 1\wedge t^{-1-d/4+\varepsilon}.
		\label{eq:fisher_5}
	\end{equation}
	
	(i). 
	Applying \eqref{eq:fisher_5} to $u$ and $w$, the Cauchy--Schwarz inequality gives for all dimensions $d\geq 1$ that
	\begin{equation*}
		|f_0(t)| \lesssim \norm{\e^{ta\Delta} {u}}_{L^2(\R^d)} \norm{\e^{ta\Delta} {w}}_{L^2(\R^d)}
		\lesssim 1\wedge t^{-2}.\label{eq:fisher_2}
	\end{equation*}
	This yields $f_0 \in L^1([0,\infty))$, proving the claim.		
	
	(ii). Lemma \ref{lem:covFun} and Lemma \ref{rescaledsemigroup}(ii) imply that
	\[ 
	\delta^{-2}\int_0^T \mathrm{Cov}(\langle X(t),u_{\delta,x}\rangle,\langle X(t),w_{\delta,x}\rangle) \diff t = \int_0^T \int_{0}^{t\delta^{-2}} f_{t,\delta,y}(t')\diff t'\diff t,\]
	with
	\begin{align}
		f_{t,\delta,y}(t')=\langle  {S}^\ast_{\theta,\delta,y}(t'){u},{S}^\ast_{\theta,\delta,y}(t'){w}\rangle_{L^2(\Lambda_{\delta,y})}\mathbf{1}(0\leq t' \leq t\delta^{-2}).\label{eq:f_t_delta}
	\end{align}
	Note that $\int_{0}^{T}\int_0^{\infty}f_0(t')\diff t'\diff t=T\psi({u},{w})$,  and write
	\begin{align*}
		& \sup_{y\in \mathcal{J}}\left| \int_0^T \int_0^{t\delta^{-2}} f_{t,\delta,y}(t')\diff t' \diff t - \int_0^T\int_0^{\infty}f_0(s)\diff t'\diff t\right|\\
		&\hspace*{1em} \leq \int_0^T \int_0^{t\delta^{-2}} \sup_{y\in \mathcal{J}}\left| f_{t,\delta,y}(t') - f_0(t')\right| \diff t'\diff t 
		+\int_0^T \int_{t\delta^{-2}}^{\infty}|f_0(t')|\diff t'\diff t\label{eq:ConvFisher_1}.
	\end{align*}
	Lemma \ref{FeynmanKac}(ii) readily yields the pointwise convergence $|f_{t,\delta,y}(t')-f_0(t')|\rightarrow 0$ as $\delta\rightarrow 0$, uniformly in $y\in \mathcal{J}$ and for any fixed $t,t'>0$. 
	Dominated convergence, i.e., \eqref{eq:fisher_5}, implies convergence to zero.
	
	(iii). Define $\bar{f}_{t,\delta,y}$ analogously to $f_{t,\delta,y}$ from \eqref{eq:f_t_delta}, now with respect to the semigroup $\bar{S}_{\delta,y}(t)$. The first step of the proof is to reduce the argument to $\bar{f}_{t,\delta,y}$. More specifically, we will show that
	\begin{equation*}
		\sup_{y\in\mathcal{J}}\int_0^T\int_0^{t\delta^{-2}}f_{t,\delta,y}(t')\diff t'\diff t =  \sup_{y\in\mathcal{J}}\int_0^T\int_0^{t\delta^{-2}}\bar{f}_{t,\delta,y}(t')\diff t'\diff t + O(\delta). \label{eq:fisher_3}
	\end{equation*}
	For doing so, consider the decomposition $f_{t,\delta,y}(t') - \bar{f}_{t,\delta,y}(t') = f^{(1)}_{t,\delta,y}(t') + f^{(2)}_{t,\delta,y}(t')$ with
	\begin{align*}
		f^{(1)}_{t,\delta,y}(t')&\coloneqq \langle (S^\ast_{\vartheta,\delta,y}(t')-\bar{S}_{\delta,y}(t'))u,S^\ast_{\vartheta,\delta,y}(t')w\rangle_{L^2(\Lambda_{\delta,y})}\mathbf{1}(0\leq t' \leq t\delta^{-2}),\\
		f^{(2)}_{t,\delta,y}(t')&\coloneqq \langle \bar{S}_{\delta,y}(t')u,(S^\ast_{\vartheta,\delta,y}(t')-\bar{S}_{\delta,y}(t'))w\rangle_{L^2(\Lambda_{\delta,y})}\mathbf{1}(0\leq t' \leq t\delta^{-2}).
	\end{align*}
	The variation of parameters formula, see p.~162 in \cite{EngNag00}, shows
	\begin{align*}
		S^\ast_{\vartheta,\delta,y}(t')-\bar{S}_{\delta,y}(t')	&=\int_{0}^{t'}\bar{S}_{\delta,y}\left(s\right)\left(A_{\vartheta,\delta,y}^\ast-\bar{A}_{\delta,y}\right)S^\ast_{\vartheta,\delta,y}\left(t'-s\right)\diff s\\
		&= -\delta\int_{0}^{t'}\bar{S}_{\delta,y}\left(s\right)\left(\theta(y+\delta\cdot)\cdot\nabla+\delta\varphi_\theta(y+\delta\cdot)\right)S^\ast_{\vartheta,\delta,y}\left(t'-s\right)\diff s.
	\end{align*} 
	Letting $\tilde{w}=\bar{S}_{\delta,y}(s) S^*_{\theta,\delta,y}(t')w$, Lemma \ref{boundS*_s} applied for $s=1/2$ gives  
	\begin{align*}
		\norm{\nabla\tilde{w}}_{L^{2}(\Lambda_{\delta,y})}
		& \lesssim \norm{(-\bar{A}_{\delta,y})^{1/2}\bar{S}_{\delta,y}(s)S^\ast_{\theta,\delta,y}(t')w}_{L^{2}(\Lambda_{\delta,y})}\lesssim (s)^{-1/2}\norm{S^\ast_{\theta,\delta,y}(t')w}_{L^2(\Lambda_{\delta,y})}.
	\end{align*} 
	Note furthermore that the adjoint of $\theta(y+\delta\cdot)\cdot\nabla$ is given by 
	\[
	-\theta(y+\delta\cdot)\cdot\nabla-\delta\varphi_\theta(y+\delta\cdot)-\delta c(y+\delta\cdot).
	\] 
	Consequently, integration by parts, the Cauchy--Schwarz inequality and \eqref{eq:fisher_5} show that, for any sufficiently small $\varepsilon>0$, $s\leq T\delta^{-2}$, uniformly in $y\in\mathcal{J}$,
	\begin{align}
		\label{eq: varparam}
		& \left|\delta^{-1}\int_0^{t\delta^{-2}}f^{(1)}_{t,\delta,y}(t')\diff t'\right| \\
		&\hspace*{1em} = \left| \int_0^{t\delta^{-2}}\int_{0}^{t'}\langle\bar{S}_{\delta,y}(s)(\theta(y+\delta\cdot)\cdot\nabla+\delta\varphi_\theta(y+\delta\cdot))S^\ast_{\theta,\delta,y}(t'-s)u,\right.S^\ast_{\theta,\delta,y}(t')w\rangle_{L^2(\Lambda_{\delta,y})} \diff s\diff t' \Bigg|\nonumber \\
		&\hspace*{1em} = \left| \int_0^{t\delta^{-2}}\int_{s}^{t\delta^{-2}}\langle\bar{S}_{\delta,y}(s)(\theta(y+\delta\cdot)\cdot\nabla+\delta\varphi_\theta(y+\delta\cdot))S^\ast_{\theta,\delta,y}(t'-s)u,\right.\nonumber S^\ast_{\theta,\delta,y}(t')w\rangle_{L^2(\Lambda_{\delta,y})} \diff t'\diff s \Bigg|\nonumber \\
		&\hspace*{1em} = \left| \int_0^{t\delta^{-2}}\int_{s}^{t\delta^{-2}}\langle S^\ast_{\theta,\delta,y}(t'-s)u,\right.\nonumber(\theta(y+\delta\cdot)\cdot\nabla-\delta c(y+\delta\cdot)) \bar{S}_{\delta,y}(s)S^\ast_{\theta,\delta,y}(t')w\rangle_{L^2(\Lambda_{\delta,y})} \diff t'\diff s \Bigg|\nonumber \\
		&\hspace*{1em} \lesssim  \int_0^{t\delta^{-2}}\int_{0}^{t\delta^{-2}} \norm{S^\ast_{\theta,\delta,y}(t')u}_{L^2(\Lambda_{\delta,y})}s^{-1/2} \norm{S^\ast_{\theta,\delta,y}(t'+s)w}_{L^2(\Lambda_{\delta,y})} \diff t' \diff s\lesssim 1.\nonumber
	\end{align}
	The bound for $f^{(2)}_{t,\delta,y}$ is obtained similarly. 
	We will conclude by proving that
	\begin{equation}
		\sup_{y\in\mathcal{J}}\int_0^T\int_0^{t\delta^{-2}}\bar{f}_{t,\delta,y}(t')\diff t'\diff t =  o(\delta). \label{eq:fisher_4}
	\end{equation}
	By Assumption \ref{ass: total}, there exists a compactly supported function $z$, given by $z=(D^{\alpha}\bar{K})/a$, such that $u=(-\bar{A})z=(-\bar{A}_{\delta,y})z$ for sufficiently small $\delta$.
	As $\bar{S}_{\delta,y}(t')$ is self-adjoint, 
	\begin{align*}
		\int_0^{t\delta^{-2}}\bar{f}_{t,\delta,y}(t')\diff t' & = \int_0^{t\delta^{-2}} \sc{\bar{S}_{\delta,y}(2t'){u}}{{w}}_{L^2(\Lambda_{\delta,y})}\diff t' \\ 
		& = \frac{1}{2}\sc{(I-\bar{S}_{\delta,y}(2t\delta^{-2}))(-\bar{A}_{\delta,y})^{-1}{u}}{{w}}_{L^2(\Lambda_{\delta,y})}\\
		& = \frac{1}{2}\sc{z}{{w}}_{L^2(\Lambda_{\delta,y})} - \frac{1}{2}\sc{\bar{S}_{\delta,y}(2t\delta^{-2})z}{{w}}_{L^2(\Lambda_{\delta,y})}.
	\end{align*}
	The first summand vanishes, as can be seen from 
	\begin{align*}
		\frac{1}{2}\sc{z}{{w}}_{L^2(\Lambda_{\delta,y})} & =  \frac{1}{2}\sc{z}{{w}}_{L^2(\R^d)}=\int_0^\infty\sc{\e^{2ta\Delta}(-a\Delta)z}{{w}}_{L^2(\R^d)}\diff t\\
		&=\int_0^{\infty}\sc{\e^{ta\Delta}{u}}{\e^{ta\Delta}{w}}_{L^2(\R^d)}\diff t=\psi({u},{w})=0.
	\end{align*}
	Consequently, \eqref{eq:fisher_4} follows from Lemma \ref{boundS*u} such that, uniformly in $y\in\mathcal{J}$,
	\begin{align*}
		\left|\int_0^T\int_0^{t\delta^{-2}}\bar{f}_{t,\delta,y}(t')\diff t'\diff t\right| & \leq  \int_0^T \frac{1}{2}|\sc{\bar{S}_{\delta,y}(2t\delta^{-2})z}{{w}}_{L^2(\Lambda_{\delta,y})}|\diff t \\
		& \lesssim \delta^2 \int_0^{T\delta^{-2}} \norm{\bar{S}_{\delta,y}(t)z}_{L^2(\Lambda_{\delta,y})}\norm{\bar{S}_{\delta,y}(t){w}}_{L^2(\Lambda_{\delta,y})}\diff t \\ 
		& \lesssim \delta^2 \int_0^{T\delta^{-2}} (1\wedge t^{-d/2-1+2\varepsilon})\diff t = O(\delta^2).
	\end{align*}
\end{proof}

\begin{proof}[Proof of Lemma \ref{ConvouterVar}]
	Using Wick's theorem (see \cite[Theorem 1.28]{janson_gaussian_1997}), write 
	\begin{equation*}
		\delta^{-6}\mathrm{Var}\left(\int_{0}^{T}\langle X(t),u_{\delta,y}\rangle\langle X(t),w_{\delta,y}\rangle \diff t\right) = 2V_1+2V_2,   
	\end{equation*}
	where $V_1=V(u,u,w,w)$, $V_2=V(u,w,w,u)$, and, for $v,v',z,z'\in L^2(\Lambda_{\delta,y})$,		
	\begin{align*}
		V(v,v',z,z')
		& =\delta^{-6} \int_0^T \int _0^t \mathrm{Cov}(\sc{X(t)}{v_{\delta,y}},\sc{X(s)}{v'_{\delta,y}})\mathrm{Cov}(\sc{X(t)}{z_{\delta,y}},\sc{X(s)}{z'_{\delta,y}})\diff s\diff t\nonumber\\
		& =\int_{0}^{T}\int_{0}^{t\delta^{-2}}\int_{0}^{t\delta^{-2}-s}f_{\delta,y}((s+r),v),(r,v')\diff r\int_{0}^{t\delta^{-2}-s}f_{\delta,y}((s+r',z),(r',z'))\diff r'\diff s\diff t,	
		\label{eq: reprV}
	\end{align*}
	with 
	\[f_{\delta,y}((l,v),(l',z))=\langle {S}^\ast_{\theta,\delta,y}(l){v},{S}^\ast_{\theta,\delta,y}(l'){z}\rangle_{L^2(\Lambda_{\delta,y})},\quad \text{ for }0\leq l,l'\leq T\delta^{-2}.
	\]
	Since the arguments for treating both terms are similar, we restrict ourselves to the upper bound for $V_1$. 
	
	(i). By the Cauchy--Schwarz inequality and \eqref{eq:fisher_5}, we find for any $\varepsilon>0$ that
	\begin{equation}
		\label{eq: boundfvar1}
		\begin{split}
			\sup_{y\in \mathcal{J}}|f_{\delta,y}((s+r,u),(r,u))|&\lesssim\sup_{y\in \mathcal{J}}\norm{ {S}^\ast_{\theta,\delta,y}(s+r){u}}_{L^2(\Lambda_{\delta,y})}\sup_{y\in \mathcal{J}}\norm{ {S}^\ast_{\theta,\delta,y}(r){u}}_{L^2(\Lambda_{\delta,y})}\\
			&\lesssim_\varepsilon (1\wedge (s+r)^{-1-d/4+\varepsilon})(1\wedge r^{-1-d/4+\varepsilon}).
		\end{split}
	\end{equation}
	Similar results are obtained for $w$. Hence,
	\begin{align*}
		\sup_{y\in \mathcal{J}}|V_1|&\lesssim\int_{0}^{T\delta^{-2}}(1\wedge s^{-2-d/2+2\varepsilon})\diff s\int_{0}^{T\delta^{-2}}(1\wedge r^{-1-d/4+\varepsilon})\diff r\int_{0}^{T\delta^{-2}}(1\wedge r'^{-1-d/4+\varepsilon})\diff r'\\
		&\lesssim1.
	\end{align*}
	
	(ii). Note that $$\sup_{y\in\mathcal{J},|y-x|\leq h}\norm{S^*_{\theta,\delta,y}(t)g^{(\theta,y,\delta)}\cdot\nabla K}_{L^2(\Lambda_{\delta,y})}\lesssim h(1\wedge t^{-d/4+\varepsilon}),$$ implying that
	\begin{align*}
		\sup_{y\in\mathcal{J},|y-x|\leq h}|f_{\delta,y}((s+r,w),(r,w))|&\lesssim h^2(1\wedge (s+r)^{-d/4+\varepsilon})(1\wedge r^{-d/4+\varepsilon}).
	\end{align*}
	Combining this with \eqref{eq: order h} and \eqref{eq: boundfvar1} gives
	\begin{flalign*}
		&\sup_{y\in \mathcal{J},|y-x|\leq h}|V_1|\lesssim \int_{0}^{T\delta^{-2}}h^2(1\wedge r^{-d/2+2\varepsilon})\diff r\lesssim h^2(1\vee \delta^{-2+d-4\varepsilon})\lesssim h^{2\beta}\delta^{-2}.
	\end{flalign*}
	
	(iii).
	The result follows similarly to part (ii), noting now that $$\sup_{y\in\mathcal{J},|y-x|\leq h}\norm{S^*_{\theta,\delta,y}(t)\varphi_\theta(y+\delta\cdot)K}_{L^2(\Lambda_{\delta,y})}\lesssim (1\wedge t^{-d/4+\varepsilon}),$$ and thus
	\begin{flalign*}
		&\sup_{y\in \mathcal{J},|y-x|\leq h}|V_1|\lesssim \int_{0}^{T\delta^{-2}}(1\wedge r^{-d/2+2\varepsilon})\diff r\lesssim (1\vee \delta^{-2+d-4\varepsilon})\lesssim h^{2\beta}\delta^{-4}.
	\end{flalign*}
\end{proof}

\subsubsection{Remaining proofs for Section \ref{sec: detailederror}}
\begin{proof}[Proof of Lemma \ref{lem: shiftsemigroup}]
	We start with deriving the following useful upper bound, which holds for any $\varepsilon>0$, and which will be applied several times:
	Lemma \ref{boundS*u} yields
	\begin{align}
		\label{eq: sharpboundg}
		\sup_{y\in\mathcal{J},|y-x|\leq h}\norm{S^*_{\theta,\delta,y}(t)g^{(\theta,\delta,y)}\cdot\nabla K}_{L^2(\Lambda_{\delta,y})}&\lesssim_\varepsilon\left(h(1\wedge t^{-1-d/4+\varepsilon})+(h^\beta+\delta)(1\wedge t^{-d/4})\right)\\
		&\leq h(1\wedge t^{-d/4}).\nonumber
	\end{align}
	Indeed, by the Minkowski inequality and \eqref{eq: functiong} with the identity function $\operatorname{id}$ on $\R^d$,
	\begin{align*}
		&\sup_{y\in\mathcal{J},|y-x|\leq h}\norm{S^*_{\theta,\delta,y}(t)g^{(\theta,\delta,y)}\cdot\nabla K}_{L^2(\Lambda_{\delta,y})}\\
		&\hspace*{1em}\lesssim\sup_{y\in\mathcal{J},|y-x|\leq h} \sum_{i=1}^d\sum_{|\alpha|=1}\bigg(|(y-x)^\alpha|\norm{S^*_{\theta,\delta,y}(t) \partial_iK}_{L^2(\R^d)}+\delta\norm{S^*_{\theta,\delta,y}(t) \operatorname{id}\partial_iK}_{L^2(\R^d)}\\
		&\hspace*{13em} +\norm{S^*_{\theta,\delta,y}(t)(D^{\alpha}(\theta_i(\xi)-\theta(x))(y+\delta\operatorname{id}-x)^\alpha)\partial_iK}_{L^2(\R^d)}\bigg)\\
		&\hspace*{1em}\lesssim h(1\wedge t^{-1-d/4+\varepsilon})+(h^\beta+\delta)(1\wedge t^{-d/4}).
	\end{align*}
	The last bound holds by three applications of Lemma \ref{boundS*u}, noting that both $\operatorname{id}\partial_iK$ and $(D^{\alpha}(\theta_i(\xi)-\theta(x))(y+\delta\operatorname{id}-x)^\alpha)\partial_iK$ are compactly supported functions with $\norm{(D^{\alpha}(\theta_i(\xi)-\theta(x))(y+\delta\operatorname{id}-x)^\alpha)\partial_iK}_{L^2(\R^d)}\lesssim h^{\beta}$ by the Hölder assumption on $\theta$. 
	Next, we study the shift from $\bar{S}_{\theta,\delta,y}(t)$ to $\e^{ta\Delta}$. 
	Assumption \ref{ass: total}, the triangle inequality and \eqref{eq: sharpboundg} imply  
	\begin{align*}
		&\sum_{k=1}^{N}w_k(x)\sc{(\bar{S}_{\delta,x_k}(2s)-\e^{2sa\Delta})\nabla K}{g^{(\theta,x_k,\delta)}\cdot \nabla K}_{L^2(\R^d)}\\
		&\hspace*{1em}\leq \sum_{k=1}^{N}|w_k(x)||\sc{\bar{S}_{\delta,x_k}(s)\nabla K}{\bar{S}_{\delta,x_k}(s)g^{(\theta,x_k,\delta)}\cdot \nabla K}_{L^2(\R^d)}|\\
		&\hspace*{3em}+\sum_{k=1}^{N}|w_k(x)||\sc{\e^{sa\Delta}\nabla K}{\e^{sa\Delta}g^{(\theta,x_k,\delta)}\cdot \nabla K}_{L^2(\R^d)}|\\
		&\hspace*{1em}\lesssim hs^{-1-d/2+2\varepsilon}.
	\end{align*}
	On the other hand, by Lemma \ref{FeynmanKac}(iii), we have for $z\in L^2(\R^d)$
	\[
	\sup_{y\in\mathcal{J}}\norm{\bar{S}_{\delta,y}(s)z-\e^{sa\Delta}z}_{L^2(\R^d)}\lesssim \delta^{1/2}s^{1/4}\e^{-\delta^{-2}s^{-1}/2}\lesssim\delta^{6+1/2}s^{3+1/4},
	\]
	using that $\e^{-x}\leq x^{-3}$ for $x>0$.
	Thus, by splitting the integral at some $r\in[0,t\delta^{-2}]$, we obtain 
	\begin{align}\label{eq: errorSbar}
		&\int_{0}^{T}\int_{0}^{t\delta^{-2}}\sum_{k=1}^{N}w_k(x)\sc{(\bar{S}_{\delta,x_k}(2s)-\e^{2sa\Delta})\nabla K}{g^{(\theta,x_k,\delta)}\cdot \nabla K}_{L^2(\R^d)}\diff s\diff t\nonumber\\
		&\hspace*{1em}\lesssim \int_{0}^{T}\int_{0}^{r}\delta^{6+1/2}s^{3+1/4}\diff s\diff t+\int_{0}^{T}\int_{r}^{t\delta^{-2}} hs^{-1-d/2+2\varepsilon}\diff s\diff t\nonumber\\
		&\hspace*{1em}\lesssim \delta^{6+1/2}\int_{0}^{r}s^{3+1/4}\diff s+h\int_{r}^{T\delta^{-2}}s^{-1-d/2+2\varepsilon} \diff s\nonumber\\
		&\hspace*{1em}\lesssim \delta^{6+1/2}r^{4+1/4}+hr^{-d/2+2\varepsilon}
	\end{align}
	for any $\varepsilon>0$. The choice $r=\delta^{-1}$ yields that the last display is of order $o(\delta^2+\delta^{d/2}).$
	We are left with the shift from $S^*_{\theta,\delta,x_k}(t)$ to $\bar{S}_{\delta,x_k}(t)$. 
	By the variation of parameters formula, cf.~\cite[p. 161]{EngNag00}, we have for $y\in\Lambda$
	\begin{align*}
		G_{\theta,\delta,y}(s)&\coloneqq S^*_{\theta,\delta,y}(s)-\bar{S}_{\delta,y}(s)
		=\int_{0}^{s}\bar{S}_{\delta,y}(r)(A^*_{\theta,\delta,y}-\bar{A}_{\delta,y})S^*_{\theta,\delta,y}(s-r)\diff r\\
		&=-\delta\int_{0}^{s}\bar{S}_{\delta,y}(r)( \theta(y+\delta\cdot)\cdot\nabla+\delta\varphi_\theta(y+\delta\cdot))S^*_{\theta,\delta,y}(s-r)\diff r.
	\end{align*}
	Consequently,
	\begin{align*}
		\begin{split}
			&\int_{0}^{T}\int_{0}^{t\delta^{-2}}\sum_{k=1}^{N}w_k(x)\sc{S^*_{\theta,\delta,x_k}(s)\nabla K}{S^*_{\theta,\delta,x_k}(s)g^{(\theta,x_k,\delta)}\cdot \nabla K}_{L^2(\Lambda_{\delta,x_k})}\diff s\diff t\\
			&\hspace*{1em}=\int_{0}^{T}\int_{0}^{t\delta^{-2}}\sum_{k=1}^{N}w_k(x)\sc{\bar{S}_{\delta,x_k}(2s)\nabla K}{g^{(\theta,x_k,\delta)}\cdot \nabla K}_{L^2(\Lambda_{\delta,x_k})}\diff s\diff t\\
			&\hspace{3em}+\int_{0}^{T}\int_{0}^{t\delta^{-2}}\sum_{k=1}^{N}w_k(x)\sc{\bar{S}_{\delta,x_k}(s)\nabla K}{G_{\theta,\delta,x_k}(s)g^{(\theta,x_k,\delta)}\cdot \nabla K}_{L^2(\Lambda_{\delta,x_k})}\diff s\diff t\\
			&\hspace{3em}+\int_{0}^{T}\int_{0}^{t\delta^{-2}}\sum_{k=1}^{N}w_k(x)\sc{G_{\theta,\delta,x_k}(s)\nabla K}{\bar{S}_{\delta,x_k}(s)g^{(\theta,x_k,\delta)}\cdot \nabla K}_{L^2(\Lambda_{\delta,x_k})}\diff s\diff t\\
			&\hspace{3em}+\int_{0}^{T}\int_{0}^{t\delta^{-2}}\sum_{k=1}^{N}w_k(x)\sc{G_{\theta,\delta,x_k}(s)\nabla K}{G_{\theta,\delta,x_k}(s)g^{(\theta,x_k,\delta)}\cdot \nabla K}_{L^2(\Lambda_{\delta,x_k})}\diff s\diff t.
		\end{split}
	\end{align*}
	The first summand in the last display has already been examined. We show the desired rate for the second summand. The bound for the other ones is obtained analogously. 
	Arguing as for \eqref{eq: varparam}, we get
	\begin{align}
		&\left|\int_{0}^{T}\int_{0}^{t\delta^{-2}}\sum_{k=1}^{N}w_k(x)\sc{\bar{S}_{\delta,x_k}(s)\nabla K}{G_{\theta,\delta,x_k}(s)g^{(\theta,x_k,\delta)}\cdot \nabla K}_{L^2(\Lambda_{\delta,x_k})}\diff s\diff t\right|\nonumber\\
		&\hspace*{1em}\lesssim \delta \int_{0}^{T}\int_{0}^{t\delta^{-2}}\int_{0}^{t\delta^{-2}}\sum_{k=1}^N|w_k(x)|\norm{\bar{S}_{\delta,x_k}(s+s')\nabla K}_{L^2(\Lambda_{\delta,x_k})}s'^{-1/2}\nonumber\\
		&\hspace{14em}\cdot\norm{S^*_{\theta,\delta,x_k}(s')g^{(\theta,x_k,\delta)}}_{L^2(\Lambda_{\delta,x_k})}\diff s\diff s'\diff t\nonumber\\
		&\hspace*{1em}\lesssim \delta \int_{0}^{T}\int_{0}^{t\delta^{-2}}\int_{0}^{t\delta^{-2}}(1\wedge (s+s')^{-1-d/4+\varepsilon})s'^{-1/2}(1\wedge s'^{-d/4+\varepsilon})h\diff s'\diff s\diff t\nonumber\\
		\label{eq: errorvariation}
		&\hspace*{1em}=O(h\delta (1\vee \delta^{-1/2+d/2-6\varepsilon}))
	\end{align}
	for any $\varepsilon>0.$
	Combining \eqref{eq: errorSbar}, \eqref{eq: errorvariation} and \eqref{eq: order h} yields the assertion.
\end{proof}

\begin{proof}[Proof of Proposition \ref{prop: splittingrest}]
	Writing for $u\in L^2(\Lambda)$ 
	\[\sc{X(t)}{u}=\sc{S_\theta(t)X_0}{u}+\sc{\bar{X}(t)}{u},\]
	we obtain the decomposition
	\begin{align*}
		\mathcal{R}_\delta^x&=\bar{\mathcal{R}}_\delta^x+\sum_{k=1}^Nw_k(x)\int_0^T\sc{\bar{X}(t)}{\nabla K_{\delta,x_k}}\sc{S_\theta(t)X_0}{(\varphi_\theta+(\theta-\theta(x))\cdot\nabla)K_{\delta,x_k}}\diff t\\
		&\hspace*{3em}+\sum_{k=1}^Nw_k(x)\int_0^T\sc{S_\theta(t)X_0}{\nabla K_{\delta,x_k}}\sc{\bar{X}(t)}{(\varphi_\theta+(\theta-\theta(x))\cdot\nabla)K_{\delta,x_k}}\diff t\\
		&\hspace*{3em}+\sum_{k=1}^Nw_k(x)\int_0^T\sc{S_\theta(t)X_0}{\nabla K_{\delta,x_k}}\sc{S_\theta(t)X_0}{(\varphi_\theta+(\theta-\theta(x))\cdot\nabla)K_{\delta,x_k}}\diff t.
	\end{align*}
	We only show that the higher order terms are of the desired order. The arguments for the lower order ones, i.e., terms containing $\varphi_\theta$ are similar and thus skipped. We hence have to show for all $1\leq i\leq d$, using the definition of $g^{(\theta,x_k,\delta)}$ in \eqref{eq: functiongtotal}, that 
	\begin{align}
		&\delta^{-2}\sum_{k=1}^Nw_k(x)\int_0^T\sc{\bar{X}(t)}{(\partial_iK) _{\delta,x_k}}\sc{S_\theta(t)X_0}{(g^{(\theta,x_k,\delta)}\cdot\nabla K)_{\delta,x_k}}\diff t\label{eq: init1}\\
		&\hspace*{3em}+\delta^{-2}\sum_{k=1}^Nw_k(x)\int_0^T\sc{S_\theta(t)X_0}{(\partial_iK)_{\delta,x_k}}\sc{\bar{X}(t)}{(g^{(\theta,x_k,\delta)}\cdot\nabla K)_{\delta,x_k}}\diff t\label{eq: init2}\\
		&\hspace*{3em}+\delta^{-2}\sum_{k=1}^Nw_k(x)\int_0^T\sc{S_\theta(t)X_0}{(\partial_iK)_{\delta,x_k}}\sc{S_\theta(t)X_0}{(g^{(\theta,x_k,\delta)}\cdot\nabla K)_{\delta,x_k}}\diff t\label{eq: init3}\\
		&\hspace*{3em}=o_{\P}(h^\beta)\nonumber
	\end{align}
	which is done by controlling the expectations and standard deviations of \eqref{eq: init1}, \eqref{eq: init2} and \eqref{eq: init3} separately for a deterministic initial condition $X_0\in L^p(\Lambda)\cap \mathcal{D}(A_\theta)$, $p>2,$ and for the stationary case $X_0=\int_{-\infty}^0S_\theta(-t')\diff W(t')$ under the extra constraint that $c-\nabla\cdot\theta\leq \gamma<0.$  
	\paragraph*{Case 1: $X_0$ is deterministic}
	Recalling \eqref{eq: order h}, the definition \eqref{eq: functiong} and the upper bound \eqref{eq: sharpboundg}, it holds for the deterministic term \eqref{eq: init3} by Lemma \ref{boundS*u}, noting furthermore $K=(-\Delta)\bar{K}$ for some $\bar{K}\in H^4(\R^d)$ with compact support, that
	\begin{align}
		&\delta^{-2}\sum_{k=1}^Nw_k(x)\int_0^T\sc{S_\theta(t)X_0}{(\partial_iK)_{\delta,x_k}}\sc{S_\theta(t)X_0}{(g^{(\theta,x_k,\delta)}\cdot\nabla K)_{\delta,x_k}}\diff t\nonumber\\
		&\hspace*{1em}=\delta^{-2}\sum_{k=1}^{N}w_k(x)\int_{0}^{T}\sc{(-A_\theta) X_0}{S^*_{\theta}(t)(-A_\theta^*)^{-1}(\partial_iK)_{\delta,x_k}}\nonumber\sc{X_0}{S^*_\theta(t)(g^{(\theta,x_k,\delta)}\cdot\nabla K)_{\delta,x_k}}\diff t\nonumber\\
		&\hspace*{1em}\lesssim\delta^{2}\int_{0}^{T\delta^{-2}}\norm{A_\theta X_0}\norm{X_0}\sup_{y\in\mathcal{J},|y-x|\leq h}\norm{S^*_{\theta,\delta,y}(t)(A_{\theta,\delta,y}^*)^{-1}\partial_iK}_{L^2(\Lambda_{\delta,y})}\cdot\norm{S^*_{\theta,\delta,y}(t)g^{(\theta,y,\delta)}\cdot\nabla K}_{L^2(\Lambda_{\delta,y})}\diff t\nonumber\\
		&\hspace*{1em}\lesssim\delta^2\int_0^{T\delta^{-2}}(1\wedge t^{-1/2-d/4+\varepsilon})\left(h(1\wedge t^{-1-d/4+\varepsilon})+(\delta+h^\beta)(1\wedge t^{-d/4})\right)\diff t\nonumber\\
		&\hspace*{1em}=o(h^\beta).
		\label{eq:detinit1}
	\end{align}
	The expectations of \eqref{eq: init1} and \eqref{eq: init2} are zero. 
	For its standard deviations, note first that, for any $y\in\mathcal{J}$ with $|y-x|\leq h$, $u,v\in L^2(\R^d)$, it holds
	\begin{align}
		&\operatorname{Var}\left(\int_0^T\sc{\bar{X}(t)}{u_{\delta,y}}\sc{X_0}{S_\theta^*(t)v_{\delta,y}}\diff t\right)\nonumber\\
		&\hspace*{1em}=2\int_0^T\int_0^t\sc{X_0}{S_\theta^*(t)v_{\delta,y}}\sc{X_0}{S_\theta^*(s)v_{\delta,y}}\operatorname{Cov}\left(\sc{\bar{X}(t)}{u_{\delta,y}},\sc{\bar{X}(s)}{u_{\delta,y}}\right)\diff s\diff t\nonumber\\
		&\hspace*{1em}=2\int_0^T\int_0^t\sc{X_0}{S_\theta^*(t)v_{\delta,y}}\sc{X_0}{S_\theta^*(s)v_{\delta,y}}\nonumber\int_0^s\sc{S^*_\theta(t-r)u_{\delta,y}}{S^*_\theta(s-r)u_{\delta,y}}\diff r\diff s\diff t\nonumber\\
		\label{eq:t1}
		&\hspace*{1em}=2\int_0^{T}\int_0^{t}\sc{X_0}{S_\theta^*(t)v_{\delta,y}}\sc{X_0}{S_\theta^*(s)v_{\delta,y}}\int_0^s\sc{S^*_\theta(t-s+r)u_{\delta,y}}{S^*_\theta(r)u_{\delta,y}}\diff r\diff s\diff t\\
		\label{eq:t2}
		&\hspace*{1em}=2\int_0^{T}\int_0^{t}\int_0^{t-s}\sc{X_0}{S_\theta^*(t)v_{\delta,y}}\sc{X_0}{S_\theta^*(t-s)v_{\delta,y}}\sc{S^*_\theta(s+r)u_{\delta,y}}{S^*_\theta(r)u_{\delta,y}}\diff r\diff s\diff t.
	\end{align}
	Applying the scaling Lemma \ref{rescaledsemigroup} to \eqref{eq:t1} with $v=\partial_iK$ and $u=g^{(\theta,y,\delta)}\cdot\nabla{K}$, followed by multiple applications of the Cauchy--Schwarz inequality and Lemma \ref{boundS*u}, thus yields by \eqref{eq: order h}
	\begin{align*}
		&\sup_{y\in\mathcal{J},|y-x|\leq h}\operatorname{Var}\left(\int_0^T\sc{S_\theta(t)X_0}{(\partial_iK)_{\delta,y}}\sc{\bar{X}(t)}{g^{(\theta,y,\delta)}\cdot\nabla K)_{\delta,y}}\diff t\right)\\            
		&\hspace*{1em}\lesssim\delta^6\int_0^{T\delta^{-2}}\left(h^2(1\wedge r^{-2-d/2+2\varepsilon})+(h^{2\beta}+\delta^2)(1\wedge r^{-d/2})\right)\diff r=o(\delta^4h^{2\beta}).
	\end{align*}
	Hence,
	\begin{align}
		&\operatorname{Var}\left(\delta^{-2}\sum_{k=1}^Nw_k(x)\int_0^T\sc{S_\theta(t)X_0}{(\partial_iK)_{\delta,x_k}}\sc{\bar{X}(t)}{(g^{(\theta,x_k,\delta)}\cdot\nabla K)_{\delta,x_k}}\diff t\right)\nonumber\\
		&\hspace*{1em}\lesssim \delta^{-4}\sum_{k=1}^N|w_k(x)|\operatorname{Var}\left(\int_0^T\sc{S_\theta(t)X_0}{(\partial_iK)_{\delta,x_k}}\sc{\bar{X}(t)}{g^{(\theta,x_k,\delta)}\cdot\nabla K)_{\delta,x_k}}\diff t\right)\nonumber\\
		&\hspace*{1em}=o(h^{2\beta}).
		\label{eq:detinit2}
	\end{align}
	Analogue calculations with $u=\partial_iK$, $v=g^{(\theta,y,\delta)}\cdot\nabla{K}$ applied to \eqref{eq:t2} also imply 
	\begin{equation}\label{eq:detinit3}
		\operatorname{Var}\left(\delta^{-2}\sum_{k=1}^Nw_k(x)\int_0^T\sc{\bar{X}(t)}{(\partial_iK) _{\delta,x_k}}\sc{S_\theta(t)X_0}{(g^{(\theta,x_k,\delta)}\cdot\nabla K)_{\delta,x_k}}\diff t\right)=o(h^{2\beta}).
	\end{equation}
	Combining \eqref{eq:detinit1}, \eqref{eq:detinit2} and \eqref{eq:detinit3} yields the claim.
	
	\paragraph*{Case 2: $X$ is stationary}
	Itô's isometry implies again that the expectations of \eqref{eq: init1} and \eqref{eq: init2} are zero, while the expected value of \eqref{eq: init3} is bounded by 
	\begin{align*}
		&\delta^{-2}\sum_{k=1}^Nw_k(x)\int_0^T\E[\sc{S_\theta(t)X_0}{(\partial_iK)_{\delta,x_k}}\sc{S_\theta(t)X_0}{(g^{(\theta,x_k,\delta)}\cdot\nabla K)_{\delta,x_k}}]\diff t\\
		&\hspace*{1em}=\delta^{-2}\sum_{k=1}^Nw_k(x)\int_0^T\int_0^\infty \langle S^*_\theta(t+t')(\partial_iK)_{\delta,x_k},S^*_\theta(t+t')(g^{(\theta,x_k,\delta)}\cdot\nabla K)_{\delta,x_k}\rangle\diff t'\diff t\\
		&\hspace*{1em}\lesssim\delta^2\int_0^{T\delta^{-2}}\int_0^\infty\sup_{y\in\mathcal{J},|y-x|\leq h}\norm{S^*_{\theta,\delta,y}(t+t')\partial_iK}_{L^2(\Lambda_{\delta,y})}\sup_{y\in\mathcal{J},|y-x|\leq h}\norm{S^*_{\theta,\delta,y}(t+t')g^{(\theta,y\delta)\cdot\nabla K}}_{L^2(\Lambda_{\delta,y})}\diff t'\diff t\\
		&\hspace*{1em}\lesssim \delta^2\int_0^{T\delta^{-2}}\left(h(1\wedge t^{-1-d/4+\varepsilon})+(\delta+h^\beta)(1\wedge t^{-d/4})\right)\diff t=o(h^\beta).
	\end{align*}
	We can bound the variance of \eqref{eq: init1} again by
	\begin{align*}
		&\operatorname{Var}\left(\delta^{-2}\sum_{k=1}^Nw_k(x)\int_0^T\sc{\bar{X}(t)}{(\partial_iK)_{\delta,x_k}}\sc{S_\theta(t)X_0}{(g^{(\theta,x_k,\delta)}\cdot\nabla K)_{\delta,x_k}}\diff t\right)\\
		&\hspace*{1em}\lesssim\delta^{-4}\sup_{y\in\mathcal{J},|y-x|\leq h}\operatorname{Var}\left(\int_0^T\sc{\bar{X}(t)}{(\partial_iK)_{\delta,y}}\sc{S_\theta(t)X_0}{(g^{(\theta,y,\delta)}\cdot\nabla K)_{\delta,y}}\diff t\right)
	\end{align*}
	similar to the deterministic case. Since $\varphi_\theta=c-\nabla\cdot\theta\leq \gamma$ for some $\gamma<0$ as assumed, the upper bound in Lemma \ref{boundS*u} holds with $\e^{-\gamma t\delta^2}$, i.e., $c_1=-\gamma$. By similar calculations as in Lemma \ref{ConvouterVar}, i.e., by Wick's Theorem, and using again Itô's isometry we get 
	\begin{align*}
		&\sup_{y\in\mathcal{J},|y-x|\leq h}\operatorname{Var}\left(\int_0^T\sc{\bar{X}(t)}{(\partial_iK)_{\delta,y}}\sc{S_\theta(t)X_0}{(g^{(\theta,y,\delta)}\cdot\nabla K)_{\delta,y}}\diff t\right)\\
		&\hspace*{1em}=2\int_0^T\int_0^t\operatorname{Cov}\left(\sc{\bar{X}(t)}{(\partial_iK)_{\delta,y}},\sc{\bar{X}(s)}{(\partial_iK)_{\delta,y}}\right)\\
		&\quad\quad\quad\quad\cdot\operatorname{Cov}\left(\sc{S_\theta(t)X_0}{(g^{(\theta,y,\delta)\cdot K)_{\delta,y}}},\sc{S_\theta(s)X_0}{(g^{(\theta,y,\delta)\cdot K)_{\delta,y}}}\right)\diff s\diff t\\
		&\hspace*{1em}=2\int_0^T\int_0^t\int_0^s\sc{S^*_\theta(t-r)(\partial_iK)_\delta,y}{S^*_\theta(s-r)(\partial_iK)_{\delta,y}}\diff r\\
		&\quad\quad\quad\quad\cdot\int_0^\infty\sc{S^*_\theta(t+r')(g^{(\theta,y\delta)}\cdot \nabla K)_{\delta,y}}{S^*_\theta(s+r')(g^{(\theta,y\delta)}\cdot \nabla K)_{\delta,y}}\diff r'\diff s\diff t\\
		&\hspace*{1em}\lesssim\delta^6\int_0^{T\delta^{-2}}(1\wedge t^{-1-d/4+\varepsilon})\diff t\int_0^{T\delta^{-2}}(1\wedge s^{-1-d/4+\varepsilon})\diff s\\
		&\quad\quad\quad\quad\cdot\int_0^\infty\left(h^2(1\wedge r^{-2-d/2+2\varepsilon})+(h^{2\beta}+\delta^2)e^{-\gamma r\delta^2}(1\wedge r^{-d/2})\right)\diff r.
	\end{align*}
	If $d\geq3$, the last display is already of order $o(\delta^4h^{2\beta})$.
	For $d\leq 2$, we bound $\e^{-\gamma r\delta^2}\lesssim r^{-1/2-\varepsilon}\delta^{-1-2\varepsilon}$, and hence
	\[ \delta^6\left(h^2+(h^{2\beta}+\delta^2)\right)\delta^{-1-2\varepsilon}=o(\delta^4h^{2\beta})
	\]
	by \eqref{eq: order h}. Similar calculations also hold for the standard deviations of \eqref{eq: init2} and \eqref{eq: init3}, implying the claim.
\end{proof}

\subsubsection{Remaining proofs for Section \ref{sec:prooflower}}
\begin{proof}[Proof of Lemma \ref{lem:concrete_lower_bound}]
	Define the integral kernels 
	\[
	\kappa_{k,l}(t)=c_{\theta^0,\delta,k,l}(t)-c_{\theta^1,\delta,k,l}(t).
	\]
	It suffices to derive the upper bound for the $L^2$-norm of $\kappa_{k,l}$, as the proof remains valid if one replaces $K_{\delta,x_k}$ by $\delta^{-4}(A^2_{\theta,\delta,x_k}K)_{\delta,x_k}$.
	This also gives the desired upper bound on the $L^2$-norm of $\kappa_{k,l}''(t)$. 
	Following the structure as in the proof of \cite[Lemma 6.10]{altmeyer_anisotrop2021}, we start by some initial notation and the diagonalizability of the semigroup $S^*_{\theta,\delta,x_k}(t)$. 
	We write $\Delta$ and $\e^{t\Delta}$ for the Laplacian and its generated semigroup on $L^2(\Lambda)$, as well as $\Delta_{\delta,x}$ and $\e^{t\Delta_{\delta,x}}$ on $L^2(\Lambda_{\delta,x})$, and $\Delta_0$ and $\e^{t\Delta_0}$ on $L^2(\R^d)$.
	We have that
	$$A^*_{\theta^1}=\Delta-\theta \cdot\nabla+(c-\nabla\cdot\theta ).$$
	Given that $\theta$ is a conservative vector field, we choose a potential $\xi$ such that $\nabla\xi(x)=\theta(x)/2$ for some function $\xi$. 
	By \cite[Example 10]{giani_2016}, $A^*_{\theta^1}$ is diagonalizable, i.e., 
	$$U_{\theta^1}^{-1}A^*_{\theta^1} U_{\theta^1} z=\Delta z+\tilde{c}_\theta z$$
	with the multiplication operator $(U_{\theta^1} z)(x)=\e^{\nabla\xi(x)}z(x)$ and $\tilde{c}_\theta=c-\frac{\nabla\cdot\theta }{2}-\frac{|\theta|^2}{4}\leq0$ due to the choice of $\theta^1$. \cite[Example 2.1 in Section II.2]{EngNag00} and the rescaling Lemma \ref{rescaledsemigroup} furthermore imply that 
	$$S^*_{\theta^1,\delta,x_k}(t)=U^{-1}_{\theta^1,\delta,x_k}\e^{t\Delta_{\delta,x_k}}U_{\theta^1,\delta,x_k}\e^{t\delta^2\tilde{c}_\theta(x_k+\delta x)},\quad S_{\theta^0,\delta,x_k}(t)=
	\e^{t\Delta_{\theta,\delta,x_k}}$$
	with $U_{\theta^1,\delta,x_k}(x)=U_{\theta^1}(x_k+\delta x)$. Note that 
	$$\e^{t\Delta}=U_{\vartheta^1}(x_k)^{-1}\e^{t\Delta}U_{\vartheta^1}(x_k).$$
	We decompose $\kappa_{k,l}=\sum_{j=1}^4\kappa ^{(j)}_{k,l}$, with
	\begin{align*}
		\kappa_{k,l}^{(1)}(t) &= \int_0^\infty\sc{U_{\vartheta^1}(x_k)^{-1}\e^{(t+t')\Delta}(U_{\vartheta^1}(x_k)-U_{\vartheta^1}\e^{\tilde{c}_{\vartheta^1}(t+t')})K_{\delta,x_k}}{\e^{t'\Delta}K_{\delta,x_l}}\diff t',\\ 
		\kappa_{k,l}^{(2)}(t) &= \int_0^\infty\sc{(U_{\vartheta^1}(x_k)^{-1}-U_{\vartheta^1}^{-1})U_{\theta^1}S^*_{\theta^1}(t+t')K_{\delta,x_k}}{\e^{t'\Delta}K_{\delta,x_l}}\diff t',\\
		\kappa_{k,l}^{(3)}(t) &= \int_0^\infty\sc{S^*_{\vartheta^1}(t+t')K_{\delta,x_k}}{U_{\vartheta^1}(x_k)^{-1}\e^{t'\Delta}(U_{\vartheta^1}(x_k)-U_{\vartheta^1}\e^{\tilde{c}_{\vartheta^1}t'})K_{\delta,x_l}}\diff t',\\
		\kappa_{k,l}^{(4)}(t) &= \int_0^\infty\sc{S^*_{\vartheta^1}(t+t')K_{\delta,x_k}}{(U_{\vartheta^1}(x_k)^{-1}-U_{\vartheta^1}^{-1})\e^{t'\Delta}U_{\vartheta^1} \e^{\tilde{c}_{\vartheta^1}t'}K_{\delta,x_l}}\diff t'.
	\end{align*}
	It suffices to show that $\sum_{1\leq k,l\leq N}\norm{\kappa_{k,l}^{(j)}}^2_{L^2([0,T])}\leq c_3\delta^8\sum_{1\leq k\leq N}(|\theta(x_k)|^2+\delta^2\tilde{c}_\theta(x_k)^2)$ for $j=1,2$. 
	The arguments for $j=3,4$ are similar and therefore skipped. Diagonal (i.e., $k=l$) and off-diagonal (i.e., $k\neq l$) terms are treated separately. 
	Set $K_{k,l}=K(\cdot+\delta^{-1}(x_k-x_l))$.
	Lemma \ref{FeynmanKac} yields
	\begin{align}
		\sup_{y\in \operatorname{supp}K}|(\e^{t\Delta_{\delta,x_k}} K_{k,l})(y)|&\lesssim \sup_{y\in \operatorname{supp}K}|(\e^{t\Delta_{0}} |K_{k,l}|)(y)|\nonumber\\
		&= \sup_{y\in \operatorname{supp}K}\int_{\R^d}(4\pi t)^{-d/2}\exp(-|x-y|^2/(4t))| K_{k,l}(x)|\,\diff x\nonumber\\
		&\leq (4\pi t)^{-d/2}\e^{-c'\frac{|x_k-x_l|^2}{\delta^2t}}\| K\|_{L^1(\R^d)}
		\lesssim t^{-d/2}\e^{-c'\frac{|x_k-x_l|^2}{\delta^2 t}},\label{eq:kappa_3}
	\end{align}
	for some $c'>0$.
	\paragraph*{Case $j=1$.}
	We start with scaling as in Lemma \ref{rescaledsemigroup} and changing variables such that, using the multiplication operators $$V_{t,t',\delta,k}(x)=1-\e^{\tilde{c}_{\vartheta^1}(x_k+\delta x)\delta^2(t+t')-\xi(x_k)+\xi(x_k+\delta x)},$$
	\begin{align}
		\kappa_{k,l}^{(1)}(t\delta^2)&= \delta^2\int_0^\infty\sc{\e^{(t+t')\Delta_{\delta,x_k}}V_{t,t',\delta,k}K}{\e^{t'\Delta_{\delta,x_k}}K_{k,l}}_{L^2(\Lambda_{\delta,x_k})}\diff t'\nonumber\\
		&= \delta^2\int_0^\infty\sc{\e^{(t/2+t')\Delta_{\delta,x_k}}V_{t,t',\delta,k}K}{\e^{(t/2+t')\Delta_{\delta,x_k}}K_{k,l}}_{L^2(\Lambda_{\delta,x_k})}\diff t'.\label{eq:kappa_4}
	\end{align}
	Since $K$ is compactly supported and $\tilde{c}_{\theta^1}\leq 0$, $V_{t,t',\delta,k}$ can be extended to smooth multiplication operators with operator norms bounded by $v_{t,t',\delta,k}= -\tilde{c}_{\vartheta^1}(x_k)\delta^2(t+t')+|\theta(x_k)|\delta$. (This can be seen from a Taylor expansion, using the Hölder smoothness assumptions for the higher order Taylor terms.) Recalling $K=\Delta^2\tilde{K}$,  Lemma \ref{boundS*u} gives, for any $\epsilon'>0$,
	\begin{align*}
		|\kappa_{k,l}^{(1)}(t\delta^2)|&\leq  \delta^2 \int_0^\infty \norm{\e^{(t/2+t')\Delta_{\delta,x_k}}V_{t,t',\delta,k}K}_{L^2(\Lambda_{\delta,x_k})}\norm{\e^{(t/2+t')\Delta_{\delta,x_k}}K_{k,l}}_{L^2(\Lambda_{\delta,x_k})}\diff t'\nonumber\\
		&\leq  \delta^2  \int_0^\infty v_{t,t',\delta,k}(1\wedge (t+t')^{-4-d/2+\epsilon'})\diff t'\nonumber\\
		&\lesssim \delta^3 (-\tilde{c}_{\vartheta^1}(x_k)\delta+ |\theta(x_k))|(1\wedge t^{-1-d/2})\nonumber\\
		&\leq \delta^3 (|\theta(x_k)|+|\delta c_\theta(x_k)|)(1\wedge t^{-1-d/2}).
	\end{align*}
	Changing variables therefore proves for the sum of diagonal terms $$\sum_{1\leq k\leq N}\norm{\kappa_{k,k}^{(1)}}^2_{L^2([0,T])}\lesssim \sum_{1\leq k\leq M}\delta^8 |\theta(x_k)|^2+\delta^{10}|c_\theta(x_k)|^2.$$
	Using Lemma \ref{boundS*u}, the integrand in \eqref{eq:kappa_4} can be bounded as follows,
	\begin{align*}
		&\sc{\e^{(t/2+t')\Delta_{\delta,x_k}}V_{t,t',\delta,k}K}{\e^{(t/2+t')\Delta_{\delta,x_k}}K_{k,l}}_{L^2(\Lambda_{\delta,x_k})}\lesssim v_{t,t',\delta,k}(1\wedge(t+t')^{-4-d/2+\varepsilon'}).
	\end{align*}
	On the other hand, using \eqref{eq:kappa_3}, it also satisfies the bound
	\begin{align*}
		\sc{\e^{(t/2+t')\Delta_{\delta,x_k}}V_{t,t',\delta,k}K}{\e^{(t/2+t')\Delta_{\delta,x_k}}K_{k,l}}_{L^2(\Lambda_{\delta,x_k})}
		&=\sc{V_{t,t',\delta,k}K}{\e^{(t+2t')\Delta_{\delta,x_k}}K_{k,l}}_{L^2(\Lambda_{\delta,x_k})}\\
		&\lesssim \norm{V_{t,t',\delta,k }K}_{L^1(\R^d)}\sup_{y\in \operatorname{supp}K}\left|( \e^{(t+2t')\Delta_{\delta,x_k}} K_{k,l})(y)\right|\\
		&\lesssim v_{t,t',\delta,k}(t')^{-d/2}\exp\left(-c'\frac{|x_k-x_l|^2}{\delta^2 t}\right).
	\end{align*}
	With respect to the off-diagonal terms, we therefore have, using the inequality $\min(a,b)\leq a^{1-\varepsilon}b^{\varepsilon}$ for $a,b\geq0$,
	\begin{align}
		\kappa_{k,l}^{(1)}(t\delta^2)&= \delta^2\int_0^\infty\sc{V_{t,t',\delta,k}K}{\e^{(t+2t')\Delta_{\delta,x_k}}K_{k,l}}_{L^2(\Lambda_{\delta,x_k})}\diff t'\nonumber\\
		&\lesssim \delta^2\int_0^\infty v_{t,t',\delta,k}^{1-\varepsilon}(1\wedge(t+t')^{-4-d/2+\epsilon'})^{1-\varepsilon}\sup_{y\in \operatorname{supp}K}\left|( \e^{(t+2t')\Delta_{\delta,x_k}} K_{k,l})(y)\right|^{\varepsilon}\diff t'\nonumber\\
		&\lesssim \delta^3(|\theta(x_k)|+\delta|\tilde{c}_\theta(x_k)|)(1\wedge t^{-1-d/2})\e^{-\varepsilon c'\frac{|x_k-x_l|^2}{\delta^2t}}.
		\label{eq:kappa_8}
	\end{align}
	Applying the bound $$\int_0^{\infty}t^{-p-1}\e^{-a/t}\diff t=a^{-p}\int_0^{\infty}t^{-p-1}\e^{-1/t}\diff t\lesssim a^{-p}$$ to $p=1+d>d$ and $a=c'\epsilon\delta^{-2}|x_k-x_l|^2$, we obtain
	\begin{equation}
		\begin{split}\label{eq:kappa_9}
			\int_0^T \kappa_{k,l}^{(1)}(t)^2\diff t &\lesssim \delta^8(|\theta(x_k)|+\delta|\tilde{c}_\theta(x_k)|)^2 \int_0^{\infty} t^{-2-d}\e^{-c'\epsilon\frac{|x_k-x_l|^2}{\delta^2 t}}\diff t\\\nonumber
			&\lesssim \frac{\delta^{10+2d}(|\theta(x_k)|^2+\delta^2|\tilde{c}_\theta(x_k)|^2)}{|x_k-x_l|^{2+2d}}.
		\end{split}
	\end{equation}
	Recalling that the $x_k$ are $\delta$-separated, we get from Lemma \ref{lem:sum:inverse:packing} below that
	\begin{align*} 
		\sum_{1\leq k\neq l\leq N}\norm{\kappa_{k,l}^{(1)}}^2_{L^2([0,T])}&\lesssim \delta^{10+2d}\sum_{k=1}^N\left(|\theta(x_k)|^2+\delta^2|\tilde{c}_\theta(x_k)|^2\right)\sum_{l=1,l\neq k}^N \frac{1}{|x_k-x_l|^{2+2d}}\\
		&\lesssim \delta^{8}\sum_{k=1}^N|\theta(x_k)|^2+\delta^2|\tilde{c}_\theta(x_k)|^2.
	\end{align*}
	Together with the bounds for the diagonal terms, this yields, for a constant $C$ depending only on $K$, 
	\[
	\sum_{1\leq k, l\leq N}\norm{\kappa_{k,l}^{(1)}}^2_{L^2([0,T])}\leq C\delta^{8}\sum_{1\leq k\leq N}|\theta(x_k)|^2+\delta^2|\tilde{c}_\theta(x_k)|^2.
	\] 
	\paragraph*{Case $j=2$.}
	As in the previous case, we have
	\begin{align*}
		&\kappa_{k,l}^{(2)}(t\delta^2)= \delta^2 \int_0^\infty \sc{(\e^{\xi(x_k+\delta x)-\xi(x_k)}-1)S^*_{\vartheta^1,\delta,x_k}(t+t')K}{\e^{t'\Delta_{\delta,x_k}}K_{k,l}}_{L^2(\Lambda_{\delta,x_k})}\diff t'.
	\end{align*}
	Using the Cauchy--Schwarz inequality, Lemma \ref{FeynmanKac}(i) and Lemma \ref{boundS*u} with $K=\Delta^2\tilde{K}$, we get for any $\epsilon>0$
	\begin{equation}
		\begin{split}\label{eq:kappa_11}
			&\sc{(\e^{\xi(x_k+\delta x)-\xi(x_k)}-1)S^*_{\vartheta^1,\delta,x_k}(t+t')K}{\e^{t'\Delta_{\delta,x_k}}K_{k,l}}_{L^2(\Lambda_{\delta,x_k})}\\
			&\hspace*{5em}\lesssim \delta |\theta(x_k)|(1\wedge (t+t')^{-2-d/4+\epsilon})\norm{|x|\e^{t'\Delta_{0}}|K_{k,l}|}_{L^2(\R^d)}.
		\end{split}
	\end{equation}
	Note that $K_{k,l}\in C^1_c(\R^d)$ such that $|K_{k,l}|\in H^{1,\infty}(\R^d)$ and $\nabla |K_{k,l}|\in L^\infty(\R^d)$ with compact support. 
	Using now \cite[Lemma A.2(ii)]{altmeyer_nonparametric_2020} to the extent that 
	\begin{align*}
		x(\e^{t'\Delta_0}|K_{k,l}|)(x) =  (\e^{t'\Delta_0}(-2t'\nabla|K_{k,l}|+x|K_{k,l}|))(x),\label{eq:kappa_10}
	\end{align*}
	we find that the $L^2(\R^d)$-norm in \eqref{eq:kappa_11} is uniformly bounded in $t'>0$. 
	Hence, $|\kappa_{k,l}^{(2)}(t\delta^2)|\lesssim \delta^3 |\theta(x_k)|(1\wedge t^{-1/2-d/4-\epsilon})$, and changing variables shows for the sum of diagonal terms $$\sum_{1\leq k\leq N}\norm{\kappa_{k,k}^{(2)}}^2_{L^2([0,T])}\lesssim \delta^{8}\sum_{1\leq k\leq N} |\theta(x_k)|^2.$$ Regarding the off-diagonal terms, we have similarly for some $\bar{K}\in L^\infty(\R^d)$ having compact support
	\begin{align*}
		&\left|\sc{(\e^{\xi(x_k+\delta x)-\xi(x_k)}-1)S^*_{\vartheta^1,\delta,x_k}(t+t')K}{\e^{t'\Delta_{\delta,x_k}}K_{k,l}}_{L^2(\Lambda_{\delta,x_k})}\right|\\
		&\hspace*{1em}=\left|\sc{K}{S_{\vartheta^1,\delta,x_k}(t+t')(\e^{\xi(x_k+\delta x)-\xi(x_k)}-1)\e^{t'\Delta_{\delta,x_k}}K_{k,l}}_{L^2(\Lambda_{\delta,x_k})}\right|\\
		&\hspace*{1em}\lesssim \delta |\theta(x_k)| \norm{K}_{L^1(\R^d)} \sup_{y\in\operatorname{supp}K} \left|\left(\e^{(t+t')\Delta_0}|x|\e^{t'\Delta_0}|K_{k,l}|\right)(y)\right|\\
		&\hspace*{1em}\lesssim \delta |\theta(x_k)| (1\vee t') \sup_{y\in\operatorname{supp}K} \left|\left(\e^{(t+2t')\Delta_0}|\bar{K}_{k,l}|\right)(y)\right|\\
		&\hspace*{1em}\lesssim \delta |\theta(x_k)| (1\vee t')t^{-d/2}\e^{-c'\frac{|x_k-x_l|^2}{\delta^2 t}},
	\end{align*}
	using \eqref{eq:kappa_3}. 
	Arguing as for \eqref{eq:kappa_8} and \eqref{eq:kappa_9}, we then find from combining the last display with \eqref{eq:kappa_11} that $|\kappa^{(4)}_{k,l}(t\delta^2)|\lesssim \delta^3 |\theta(x_k)| t^{-1/2-d/4-\epsilon'}\e^{-\epsilon c' \frac{|x_k-x_l|^2}{\delta^2 t}}$ for some $\epsilon,\epsilon'>0$ and 
	\[
	\int_0^T \kappa_{k,l}^{(2)}(t)^2\diff t \lesssim \frac{\delta^{8+d+4\epsilon}|\theta(x_k)|}{|x_k-x_l|^{4\epsilon+d}}.
	\]
	So, all in all, for diagonal and off-diagonal terms,  $$\sum_{1\leq k, l\leq N}\norm{\kappa_{k,l}^{(2)}}^2_{L^2([0,T])}\leq C\delta^8 \sum_{1\leq k\leq N}|\theta(x_k)|^2,$$ for a constant $C$ depending only on $K$.
\end{proof}

\begin{lemma}[Lemma A.3 in \cite{altmeyer_anisotrop2021}]
	\label{lem:sum:inverse:packing}
	Let $x_1,\dots,x_N$ be $\delta$-separated points in $\R^d$, and let $p>d$. Then, for a constant $C=C(d,p)$,
	\begin{align*}
		\sum_{k=2}^N\frac{1}{|x_1-x_k|^p}\le C\delta^{-p}.
	\end{align*}
\end{lemma}

\paragraph*{Acknowledgments}
We gratefully acknowledge financial support provided by the Carlsberg Foundation Young Researcher Fellowship grant CF20-0604 ``Exploring the potential of nonparametric modelling of complex systems via SPDEs''.

\bibliography{references2}
\bibliographystyle{apalike2}
\end{document}